\documentclass[12pt]{amsart}
\usepackage{simplemargins}
 \setallmargins{1in}
\usepackage{amscd,amsmath,amssymb,epsfig,amsthm,latexsym}
\usepackage{graphicx}
\usepackage{tikz}
\usepackage{multirow}
\usepackage[enableskew,vcentermath]{youngtab}
\usepackage{hyperref}
\usepackage{tabu}
\hypersetup{colorlinks=false}
\newtheorem{theorem}{Theorem}[section]
\newtheorem{proposition}[theorem]{Proposition}
\newtheorem{corollary}[theorem]{Corollary}
\newtheorem{conjecture}[theorem]{Conjecture}
\newtheorem{definition}[theorem]{Definition}

\newtheorem{lemma}[theorem]{Lemma}
\newtheorem{remark}[theorem]{Remark}

\usepackage{url}
\usepackage{array}
\usepackage{varwidth}
\usepackage{tikz}
\usetikzlibrary{patterns}
\usepackage{multirow}
\usepackage{textcomp}
\usepackage{bbm}
\usepackage{color}
\usepackage{amsthm}
\usepackage{algorithm}
\usepackage{algorithmic}

\usetikzlibrary{patterns}
\newcommand{\greatcircle}[5][]{%
\path[#1,rotate=#5] (#2) circle [x radius=#3, y radius=#4];

\begin{scope}[rotate=#5]
\clip (#3,0) rectangle ([xshift=-0.1,yshift=-0.1]-#3,-#4);
\draw[#1] (#2) circle [x radius=#3, y radius=#4];
\end{scope}

\begin{scope}[rotate=#5, yscale=-1,xscale=-1]
\clip (#3,0) rectangle ([xshift=-0.1,yshift=-0.1]-#3,-#4);
\draw[#1, dashed] (#2) circle [x radius=#3, y radius=#4];
\end{scope}

}

\DeclareMathOperator{\vc}{\mathbf{c}}

\DeclareMathOperator{\vx}{\mathbf{x}}
\DeclareMathOperator{\vv}{\mathbf{v}}
\DeclareMathOperator{\vy}{\mathbf{y}}
\DeclareMathOperator{\vz}{\mathbf{z}}
\DeclareMathOperator{\vw}{\mathbf{w}}

\DeclareMathOperator{\mcA}{\mathcal{A}}
\DeclareMathOperator{\mcC}{\mathcal{C}}

\DeclareMathOperator{\mcH}{\mathcal{H}}
\DeclareMathOperator{\mcL}{\mathcal{L}}
\DeclareMathOperator{\mcM}{\mathcal{M}}
\DeclareMathOperator{\mcP}{\mathcal{P}}
\DeclareMathOperator{\mcT}{\mathcal{T}}

\DeclareMathOperator{\mcV}{\mathcal{V}}

\DeclareMathOperator{\diam}{diam}
\DeclareMathOperator{\sgn}{sign}

\DeclareMathOperator{\rank}{rank}

\DeclareMathOperator{\bbR}{\mathbb{R}}
\DeclareMathOperator{\bbS}{\mathbb{S}}

\DeclareMathOperator{\zerovec}{\mathbf{0}}

\DeclareMathOperator{\sign}{sign}
\DeclareMathOperator{\supp}{supp}

\newcommand\commentout[1]{}

\title[Diameter of Oriented Matroids]{Diameters of Cocircuit Graphs of Oriented Matroids: An Update} 

\author[I.\,Adler]{Ilan Adler}
\address{Dept. of Industrial Engineering and Operations Research, Univ. of California, Berkeley}
\email[I.\,Adler]{adler@ieor.berkeley.edu}

\author[J.\,A. De Loera]{Jes\'us A. De Loera}
\address{Department of Mathematics, Univ. of California, Davis}
\email[J.\,A. De Loera]{deloera@math.ucdavis.edu}

\author[S.\,Klee]{Steven Klee}
\address{Department of Mathematics, Seattle University}
\email[S.\,Klee]{klees@seattleu.edu}

\author[Z.\,Zhang]{Zhenyang Zhang}
\address{Department of Mathematics, Univ. of California, Davis}
\email[Z.\,Zhang]{zhenyangz@math.ucdavis.edu}

\date{\today}
\thanks{ \textit{Keywords and phrases}. Oriented matroids, diameter of graphs, geometric complexity, Polynomial Hirsch conjecture, simplex method, pivoting algorithms.}
\begin{document}
    \maketitle

\begin{abstract}
Oriented matroids (often called order types) are combinatorial structures that generalize point configurations, vector configurations, hyperplane arrangements, polyhedra, linear programs, and directed graphs. Oriented matroids have played a key  role in combinatorics, computational geometry, and optimization. This paper surveys prior work and presents an update on the search for bounds on the diameter of the cocircuit graph of an oriented matroid. 

We review the diameter problem and show the diameter bounds of general oriented matroids reduce to those of uniform oriented matroids. We give the latest exact bounds for oriented matroids of low rank and low corank, and for all oriented matroids with up to nine elements (this part required a large computer-based proof). The motivation for our investigations is the complexity of the \emph{simplex method} and the \emph{criss-cross method}.  For arbitrary oriented matroids, we present an improvement to a quadratic bound of Finschi. Our discussion highlights two very important conjectures related to the polynomial Hirsch conjecture for polytope diameters. 
\end{abstract}
 
\section{Introduction}

Oriented matroids are combinatorial structures that generalize many types of objects, including point configurations, 
vector configurations, hyperplane arrangements, polyhedra, linear programs, and directed graphs. Oriented matroids have 
played a key role in combinatorics, geometry, and optimization (see Bj\"{o}rner et al. \cite{OMBook} and Ziegler \cite{Ziegler-book}). 
An important family, \emph{realizable} oriented matroids, are given by hyperplane arrangements. In this case, 
the cocircuit graph is just the one-skeleton of the cell complex obtained by intersecting a central hyperplane arrangement  with a unit sphere. Characterizations of cocircuit graphs have been explored in \cite{Felsner-et-al, MB-Strausz}.  In this article we are interested instead in bounding the diameter of the cocircuit graph of an oriented matroid.
The motivation for our investigations is the complexity of the \emph{simplex method} \cite{bertsimastsitsiklis, Schrijver1986}
and of the \emph{criss-cross method} \cite{fukudaterlaky1, fukudaterlaky2}.  Both algorithms are pivoting methods that jump 
from cocircuit to cocircuit, using edges of the cocircuit graph. The following conjecture is the oldest and the most ambitious challenge about 
the diameter today.

\begin{conjecture} \label{conj:mainproblem}
Let $\mathcal{M}$ be an oriented matroid of rank $r$ on $n$ elements, and let $G^*(\mathcal{M})$ be its cocircuit graph.  
Then  $\diam(G^*(\mathcal{M})) \leq n-r+2.$
\end{conjecture}

Conjecture \ref{conj:mainproblem} bears a striking resemblance to the famous Hirsch conjecture for convex polytopes, which was disproved by Santos \cite{Santos}, and with good reason. Let $P \subseteq \bbR^d$ be a $d$-polytope defined by $n$ hyperplane inequalities. Lifting $P$ to $\bbR^{d+1}$ (and setting $r = d+1$) determines a central hyperplane arrangement in $\bbR^r$, one of whose cones is the nonnegative span of  $P$. Therefore, $P$ gives rise to an oriented matroid $\mcM$ whose cocircuit graph contains the graph of $P$ as an induced subgraph (see Figure \ref{fig:lifting}).

\begin{figure}[h]
\begin{minipage}{.31\textwidth}
\begin{tikzpicture}[scale=.4]
\draw[lightgray!50, fill=lightgray!50] (0,0) -- (0,1) -- (2,3) -- (7,1) -- (5,0) -- (0,0);
\draw[<->] (-1,0) -- (8,0);
\draw[<->] (0,-1) -- (0,5);
\draw[<->] (-1,0) -- (4,5);
\draw[<->] (3,-1) -- (8,1.5);
\draw[<->] (-1,21/5) -- (8.5,.4);
\draw[fill=black] (0,0) circle (.1);
\draw[fill=black] (0,1) circle (.1);
\draw[fill=black] (2,3) circle (.1);
\draw[fill=black] (7,1) circle (.1);
\draw[fill=black] (5,0) circle (.1);
\end{tikzpicture}
\end{minipage}
\begin{minipage}{.31\textwidth}
\begin{tikzpicture}[scale=.4]
\draw[lightgray!50, fill=lightgray!50] (0,5,0) -- (0,5,1) -- (2,5,3) -- (7,5,1) -- (5,5,0) -- (0,5,0);
\draw[<->] (-1,5,0) -- (8,5,0);
\draw[<->] (0,5,-1) -- (0,5,5);
\draw[<->] (-1,5,0) -- (4,5,5);
\draw[<->] (3,5,-1) -- (8,5,1.5);
\draw[<->] (-1,5,21/5) -- (8.5,5,.4);
\draw[fill=black] (0,5,0) circle (.1);
\draw[fill=black] (0,5,1) circle (.1);
\draw[fill=black] (2,5,3) circle (.1);
\draw[fill=black] (7,5,1) circle (.1);
\draw[fill=black] (5,5,0) circle (.1);

\draw[dotted] (0,0,0) -- (0,5,0);
\draw (0,0,0) -- (0,5,1);
\draw (0,0,0) -- (2,5,3);
\draw (0,0,0) -- (7,5,1);
\draw[dotted] (0,0,0) -- (5,5,0);

\end{tikzpicture}
\end{minipage}
\begin{minipage}{.33\textwidth}
\begin{tikzpicture}[scale=.4]

 \draw[blue, fill=gray!50, opacity=.5]
(-3.716,-0.669) -- 
(-3.657,-0.682) -- 
(-3.597,-0.695) -- 
(-3.536,-0.707) -- 
(-3.473,-0.719) -- 
(-3.41,-0.731) -- 
(-3.346,-0.743) -- 
(-3.28,-0.755) -- 
(-3.214,-0.766) -- 
(-3.147,-0.777) -- 
(-3.078,-0.788) -- 
(-3.009,-0.799) -- 
(-2.939,-0.809) -- 
(-2.868,-0.819) -- 
(-2.796,-0.829) -- 
(-2.723,-0.839) -- 
(-2.65,-0.848) -- 
(-2.575,-0.857) -- 
(-2.5,-0.866) -- 
(-2.424,-0.875) -- 
(-2.347,-0.883) -- 
(-2.27,-0.891) -- 
(-2.192,-0.899) -- 
(-2.113,-0.906) -- 
(-2.034,-0.914) -- 
(-1.954,-0.921) -- 
(-1.873,-0.927) -- 
(-1.792,-0.934) -- 
(-1.71,-0.94) -- 
(-1.712,-0.919)--
(-1.634,-0.954)--
(-1.557,-0.989)--
(-1.479,-1.023)--
(-1.4,-1.057)--
(-1.321,-1.091)--
(-1.242,-1.125)--
(-1.162,-1.158)--
(-1.082,-1.191)--
(-1.001,-1.223)--
(-0.921,-1.255)--
(-0.84,-1.287)--
(-0.758,-1.318)--
(-0.677,-1.349)--
(-0.595,-1.379)--
(-0.513,-1.41)--
(-0.431,-1.439)--
(-0.349,-1.468)--
(-0.266,-1.497)--
(-0.184,-1.525)--
(-0.102,-1.553)--
(-0.019,-1.581)--
(0.063,-1.607)--
(0.146,-1.634)--
(0.228,-1.66)--
(0.311,-1.685)--
(0.393,-1.71)--
(0.475,-1.734)-- 
(0.473,-1.628)--
(0.47,-1.71)--
(0.467,-1.792)--
(0.464,-1.873)--
(0.46,-1.954)--
(0.457,-2.034)--
(0.453,-2.113)--
(0.449,-2.192)-- 
(0.461,-2.167)--
(0.391,-2.216)--
(0.321,-2.264)--
(0.25,-2.312)--
(0.18,-2.359)--
(0.11,-2.405)--
(0.039,-2.45)--
(-0.031,-2.495)--
(-0.102,-2.539)--
(-0.172,-2.582)--
(-0.243,-2.624)--
(-0.313,-2.666)--
(-0.383,-2.707)--
(-0.454,-2.747)--
(-0.524,-2.786)--
(-0.593,-2.824)--
(-0.663,-2.862)--
(-0.733,-2.898)--
(-0.802,-2.934)-- 
(-0.804,-2.957)--
(-0.883,-2.921)--
(-0.961,-2.883)--
(-1.039,-2.845)--
(-1.117,-2.806)--
(-1.194,-2.766)--
(-1.271,-2.725)--
(-1.348,-2.684)--
(-1.424,-2.641)--
(-1.5,-2.598)--
(-1.575,-2.554)--
(-1.65,-2.509)--
(-1.725,-2.464)--
(-1.798,-2.417)--
(-1.872,-2.37)--
(-1.944,-2.323)--
(-2.017,-2.274)--
(-2.088,-2.225)--
(-2.159,-2.175)--
(-2.229,-2.124)--
(-2.299,-2.073)--
(-2.368,-2.022)--
(-2.436,-1.969)--
(-2.503,-1.916)--
(-2.57,-1.863)--
(-2.635,-1.808)--
(-2.7,-1.754)--
(-2.765,-1.698)--
(-2.828,-1.643)--
(-2.891,-1.586)--
(-2.952,-1.53)--
(-3.013,-1.472)--
(-3.073,-1.415)--
(-3.132,-1.357)--
(-3.189,-1.298)--
(-3.246,-1.239)--
(-3.302,-1.18)--
(-3.357,-1.12)--
(-3.411,-1.06)--
(-3.464,-1)--
(-3.516,-0.939)--
(-3.567,-0.878)--
(-3.616,-0.817)--
(-3.665,-0.756)--
(-3.712,-0.694);
\draw[black](0,0) circle [radius=5cm];
\greatcircle[blue] {0,0} {5cm}{1cm}{0} 
\greatcircle[blue] {0,0}{5cm}{.5cm}{90}
\greatcircle[blue] {0,0}{5cm}{3cm}{-30}
\greatcircle[blue] {0,0}{5cm}{2cm}{40}
\greatcircle[blue] {0,0}{5cm}{1.5cm}{-20}
 \draw[fill=black] (-3.716,-0.669) circle (.1);
 \draw[fill=black] (-1.71,-0.94) circle (.1);
 \draw[fill=black] (0.475,-1.734) circle (.1);
 \draw[fill=black] (0.449,-2.192) circle (.1);
 \draw[fill=black] (-0.802,-2.934) circle (.1);

 \draw[fill=black] (3.716,0.669) circle (.1);
 \draw[fill=black] (1.71,0.94) circle (.1);
 \draw[fill=black] (-0.475,1.734) circle (.1);
 \draw[fill=black] (-0.449,2.192) circle (.1);
 \draw[fill=black] (0.802,2.934) circle (.1);

 \draw[blue, fill=gray!50, opacity=.5, dashed]
(3.716,0.669) --  
(3.657,0.682) --  
(3.597,0.695) --  
(3.536,0.707) --  
(3.473,0.719) --  
(3.41,0.731) --  
(3.346,0.743) --  
(3.28,0.755) --  
(3.214,0.766) --  
(3.147,0.777) --  
(3.078,0.788) --  
(3.009,0.799) --  
(2.939,0.809) --  
(2.868,0.819) --  
(2.796,0.829) --  
(2.723,0.839) --  
(2.65,0.848) --  
(2.575,0.857) --  
(2.5,0.866) --  
(2.424,0.875) --  
(2.347,0.883) --  
(2.27,0.891) --  
(2.192,0.899) --  
(2.113,0.906) --  
(2.034,0.914) --  
(1.954,0.921) --  
(1.873,0.927) --  
(1.792,0.934) --  
(1.71,0.94) --  
(1.712,0.919) --
(1.634,0.954) --
(1.557,0.989) --
(1.479,1.023) --
(1.4,1.057) --
(1.321,1.091) --
(1.242,1.125) --
(1.162,1.158) --
(1.082,1.191) --
(1.001,1.223) --
(0.921,1.255) --
(0.84,1.287) --
(0.758,1.318) --
(0.677,1.349) --
(0.595,1.379) --
(0.513,1.41) --
(0.431,1.439) --
(0.349,1.468) --
(0.266,1.497) --
(0.184,1.525) --
(0.102,1.553) --
(0.019,1.581) --
(-0.063,1.607) --
(-0.146,1.634) --
(-0.228,1.66) --
(-0.311,1.685) --
(-0.393,1.71) --
(-0.475,1.734) -- 
(-0.473,1.628) --
(-0.47,1.71) --
(-0.467,1.792) --
(-0.464,1.873) --
(-0.46,1.954) --
(-0.457,2.034) --
(-0.453,2.113) --
(-0.449,2.192) -- 
(-0.461,2.167) --
(-0.391,2.216) --
(-0.321,2.264) --
(-0.25,2.312) --
(-0.18,2.359) --
(-0.11,2.405) --
(-0.039,2.45) --
(-0.031,2.495) --
(0.102,2.539) --
(0.172,2.582) --
(0.243,2.624) --
(0.313,2.666) --
(0.383,2.707) --
(0.454,2.747) --
(0.524,2.786) --
(0.593,2.824) --
(0.663,2.862) --
(0.733,2.898) --
(0.802,2.934) -- 
(0.804,2.957) --
(0.883,2.921) --
(0.961,2.883) --
(1.039,2.845) --
(1.117,2.806) --
(1.194,2.766) --
(1.271,2.725) --
(1.348,2.684) --
(1.424,2.641) --
(1.5,2.598) --
(1.575,2.554) --
(1.65,2.509) --
(1.725,2.464) --
(1.798,2.417) --
(1.872,2.37) --
(1.944,2.323) --
(2.017,2.274) --
(2.088,2.225) --
(2.159,2.175) --
(2.229,2.124) --
(2.299,2.073) --
(2.368,2.022) --
(2.436,1.969) --
(2.503,1.916) --
(2.57,1.863) --
(2.635,1.808) --
(2.7,1.754) --
(2.765,1.698) --
(2.828,1.643) --
(2.891,1.586) --
(2.952,1.53) --
(3.013,1.472) --
(3.073,1.415) --
(3.132,1.357) --
(3.189,1.298) --
(3.246,1.239) --
(3.302,1.18) --
(3.357,1.12) --
(3.411,1.06) --
(3.464,1) --
(3.516,0.939) --
(3.567,0.878) --
(3.616,0.817) --
(3.665,0.756) --
(3.712,0.694);
\end{tikzpicture}
\end{minipage}
\caption{A polytope in $\mathbb{R}^2$ (left), its lifting to $\mathbb{R}^3$ (center), and the intersection with the resulting hyperplane arrangement on $\mathbb{S}^2$ (right).}
\label{fig:lifting}
\end{figure}
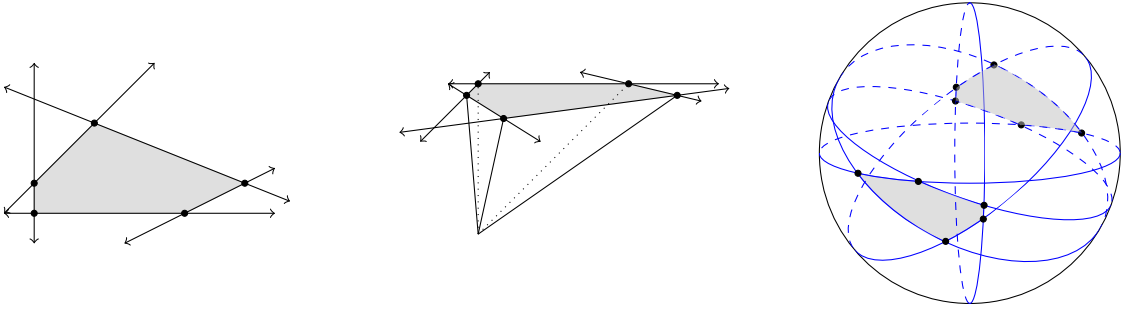

Substituting $r=d+1$ in Conjecture \ref{conj:mainproblem} gives an upper bound of $n-r+2 = n-d+1$, which differs from the conjectured Hirsch bound by $1$.  The reason for this is that each signed cocircuit $X$ has an antipodal cocircuit $-X$.  We will see later that when $\mcM$ is uniform, the distance between antipodal cocircuits is exactly $n-r+2$. 

Conjecture \ref{conj:mainproblem} has appeared in the literature in several forms. Babson, Finschi, and Fukuda \cite[Lemma 6]{Babson-Finschi-Fukuda} established Conjecture \ref{conj:mainproblem} for uniform oriented matroids of rank $2$ and rank $3$, showing further that only antipodal cocircuits can have distance $n-r+2$.  Felsner et al. \cite[Lemma 4.1]{Felsner-et-al} also showed that the conjecture is true for uniform oriented matroids with rank at most 3 and stated again the famous Conjecture \ref{conj:mainproblem} in \cite[Question 4.2]{Felsner-et-al} with a strong emphasis on the important role of antipodal cocircuits. Finschi \cite[Open Problem 5]{Finschi-thesis} asked whether $\diam(G^*(\mcM)) \leq c \cdot n$ for some constant $c$ that is independent of $n$ and $r$.  Fukuda (personal communication) noted to us that Conjecture \ref{conj:mainproblem} is a folklore problem that predates all these papers.

Aside from the results of Babson, Finschi, and Fukuda in low rank, the most general progress that has been made on Conjecture \ref{conj:mainproblem} seems to come from Finschi's thesis.

\begin{theorem}{\rm{(Finschi \cite[Proposition 2.6.1]{Finschi-thesis})}} \label{thm:Finschi-bound}

\noindent Let $\mcM$ be a uniform oriented matroid of rank $r$ on $n$ elements. Then
$$
\diam(G^*(\mcM)) \leq n-r+2 + \sum_{k=1}^{\min(r-2,n-r)}\left(\left\lfloor \frac{n-r-k}{2}\right\rfloor+1\right).
$$
\end{theorem}

The bound in Theorem \ref{thm:Finschi-bound} is tight when $r=2$ or $r=n$, but in general it is not. 

\subsection{Notation and Definitions}\hfill

We use standard notation about oriented matroids from Ziegler \cite{Ziegler-book} and the classic book of Bj\"{o}rner et al. \cite{OMBook}.  
The geometric intuition that accompanies these definitions will be discussed at greater length in Section \ref{section:overview}, but
for now we introduce the minimal notation and definitions so that we can state our results formally.

A purely combinatorial description of oriented matroids can be given in terms of special \emph{sign vectors}.
If $E$ is a finite set, we use $\{+,-,0\}^E$ to denote the set of all vectors of signs, with entries indexed by the elements of $E$. We will use capital letters $X,Y,Z,\ldots$ to represent elements of $\{+,-,0\}^E$ and subscripts $X_e$ to reference the entry of $X$ indexed by the element $e \in E$.  We can always negate a sign vector: if $X = (X_e: e \in E)$, then $-X = (-X_e: e \in E)$.  

The \textit{positive}, \textit{negative}, and \textit{zero} parts of a sign vector $X \in \{+,-,0\}^E$ are defined respectively as $X^+ = \{e \in E : X_e = +\}$, $X^- = \{e \in E : X_e = -\}$, and $X^0 = \{e \in E : X_e = 0\}$.  The \textit{support} of $X$ is defined as $\supp(X) = X^+ \cup X^-$. If $X$ and $Y$ are sign vectors, their \textit{separating set} is $S(X,Y) = (X^+ \cap Y^-) \cup (X^- \cap Y^+)$, and their \textit{composition} is the sign vector $X \circ Y$ whose entries are given by 
$$
(X \circ Y)_e = \begin{cases}
X_e & \text{ if } X_e \neq 0, \\
Y_e & \text{ otherwise.}
\end{cases}
$$

For the moment we will only provide the \emph{cocircuit axioms} of oriented matroids. As with classical matroids, there are also several cryptomorphic  definitions of oriented matroids; see \cite{OMBook} for more details. We will briefly introduce some of these details later. The cocircuits and covectors of an oriented matroid are special types of sign vectors that satisfy certain axioms:

\begin{definition}\label{OM_axioms}
An oriented matroid $\mcM = (E,\mcC^*)$ consists of a finite set $E$ and a subset $\mcC^* \subseteq \{+,-,0\}^E$, called signed cocircuits, that satisfy the following conditions. 
\begin{enumerate}
\item[{\rm (CC0)}] $\mathbf{0} \notin \mathcal{C}^*$;
\item[{\rm (CC1)}] if $X \in \mathcal{C}^*$, then $-X \in \mathcal{C}^*$;
\item[{\rm (CC2)}] for all $X,Y \in \mathcal{C}^*$, if $\supp(X) \subseteq \supp(Y)$, then $X = Y$ or $X = -Y$; and
\item[{\rm (CC3)}] if $X,Y \in \mathcal{C}^*$, $X \neq -Y$, and $e \in S(X,Y)$, then there exists $Z \in \mathcal{C}^*$ such that $Z^+ \subseteq (X^+ \cup Y^+) \setminus \{e\}$ and $Z^- \subseteq (X^- \cup Y^-) \setminus \{e\}$.
\end{enumerate}
\end{definition}

Given an oriented matroid $\mcM$, we can consider the set $\mcL = \{X^0 \circ X^1 \circ \cdots \circ X^k\ : \ X^i \in \mcC^*(\mcM)\}$ of all possible signed \emph{covectors}, obtained by successively composing signed cocircuits.  The set $\mcL$ has a natural poset structure, which we denote by $\Gamma(\mcL)$ (in fact, $\Gamma(\mcL)$ is a graded lattice). The order is obtained from the component-wise partial order on vectors in $\{+,-,0\}^E$ with $0 < +,-$. We will revisit this poset later  in a geometric setting. 

The \textit{rank} of $\mcM$ is defined to be one less than the length of the longest chain of elements in the poset $\Gamma(\mcL)$. Again, this is not the only way to define the rank. We say an element of $E$ is a \emph{coloop} if it is not present in the support of any signed cocircuit. For brevity, signed cocircuits will also be called cocircuits. It is well known that every matroid has a \emph{dual matroid}. In the case of oriented matroids, this concept is more delicate, but there is also a notion of duality. One can then talk about \emph{circuits}, which are the cocircuits of the dual oriented matroid, and the related notions of \emph{corank}, \emph{loops}, etc. The corank of an oriented matroid on $n$ elements of rank $r$ is $n-r$.

The \textit{cocircuit graph} of an oriented matroid $\mcM$ of rank $r$ is the graph $G^*(\mcM)$ whose vertices are the signed cocircuits of $\mcM$, with an edge connecting signed cocircuits $X$ and $Y$ if $|X^0 \cap Y^0| \geq r-2$ and $S(X, Y) = \emptyset$.  An oriented matroid is \textit{uniform} if $|X^0| = r-1$ for every cocircuit $X \in \mathcal{C}^*$. If $X$ and $Y$ are signed cocircuits in $\mcM$, we use $d_{\mcM}(X,Y)$ to denote the distance from $X$ to $Y$ in $G^*(\mcM)$; that is, the length of the shortest path from $X$ to $Y$ in $G^*(\mcM)$. We call a path $P$ from $X$ to $Y$ \emph{crabbed} (first introduced in \cite{MB-Strausz}), if for every cocircuit $W \in P$, $W^+ \subseteq X^+ \cup Y^+$ and $W^- \subseteq X^- \cup Y^-$. The \textit{diameter} of $G^*(\mcM)$ is defined as $\diam(G^*(\mcM)) = \max\{d_{\mcM}(X,Y) : X,Y \in \mcC^*(\mcM)\}$.

\subsection{Our Results}\hfill

We begin with an overview of our results.  

One of the first reductions made in studying the Hirsch conjecture was given by Klee and Walkup \cite{Klee-Walkup}, who showed it was sufficient to study \emph{simple} polytopes.  These are $d$-polytopes in which each vertex is supported by exactly $d$ facets.  
We make a similar reduction  from arbitrary to \textit{uniform} oriented matroids.

\begin{lemma}\label{perturbation-lemma}
Let $\mathcal{M}$ be an oriented matroid of rank $r$ on $n$ elements.  Then there exists a uniform oriented matroid 
$\mathcal{M}'$ of rank $r$ on $n$ elements such that 
$$
\diam(G^*(\mathcal{M})) \leq \diam(G^*(\mathcal{M}')).
$$
Moreover, when $\mcM$ is realizable, then $\mcM'$ can be taken to be realizable as well. 
\end{lemma} 

Therefore, for the purposes of studying Conjecture \ref{conj:mainproblem}, it suffices to consider only uniform oriented matroids.

The following lemma is small but powerful, because it shows the discrepancy between the diameter given in Conjecture \ref{conj:mainproblem} and the classical Hirsch bound cannot be improved.

\begin{lemma}\label{lem:distance-lower-bound}
Let $\mathcal{M}$ be a uniform oriented matroid of rank $r$ on $n$ elements, and let $X,Y \in \mathcal{C}^*(\mathcal{M})$.  Then
\begin{equation} \label{eq:distbound}
d_{\mcM}(X,Y) \geq 
\begin{cases}
|S(X,Y)| + |X^0 \setminus Y^0| & \text{ if } X \neq -Y,\\
n-r+2 & \text{ if } X = -Y.
\end{cases}
\end{equation}
Moreover, if $|X^0 \setminus Y^0| \leq1$, then the  inequality  \eqref{eq:distbound} holds with equality: $d_{\mcM}(X,Y) = 1+|S(X,Y)|$,
and in particular, when $X=-Y$, then 
$d_{\mcM}(X,Y) = n-r+2$. 
\end{lemma}

One could hope that  $d_{\mcM}(X,Y) \leq n-r+1$ when $X$ and $Y$ are not antipodal cocircuits.  Finschi posed a similar question in his thesis \cite[Open Problem 2]{Finschi-thesis}, as did Felsner et al. \cite[Question 4.2]{Felsner-et-al}.
However, one can show this is false by considering Santos's counterexample to the Hirsch conjecture, as we will discuss in Section 3.

\begin{proposition}\label{santos-diameter}
There is a uniform oriented matroid $\mcM$ of rank 21 on 40 elements that has a pair of non-antipodal cocircuits $X$ and $Y$ such that $d_{\mcM}(X,Y) \geq 21 = n-r+2.$
\end{proposition}

Next, we turn our attention to small oriented matroids, for which $n, r$ or $n - r$ are small.

\begin{theorem}\label{thm:lowrankcorank}
Let $\mathcal{M}$ be a uniform oriented matroid of rank $r$ on $n$ elements.
\begin{enumerate}
\item[a.]  If $n \leq 9$, then $\diam(G^*(\mathcal{M})) = n-r+2.$
\item[b.]  If $r \leq 3$, then $\diam(G^*(\mathcal{M})) = n-r+2.$
\item[c.] If $n-r \leq 4$, then $\diam(G^*(\mathcal{M})) = n-r+2.$
\end{enumerate}
  \end{theorem}

Babson, Finschi, and Fukuda \cite[Lemma 6]{Babson-Finschi-Fukuda} and Felsner et al. \cite[Lemma 4.1]{Felsner-et-al} gave proofs of Conjecture \ref{conj:mainproblem} for $r \leq 3$.  We give a new geometric proof in rank three and add new results in low corank in Section 4.



We conclude with a quadratic bound on the diameter of the cocircuit graph of an oriented matroid. We modify Finschi's 
proof of Theorem \ref{thm:Finschi-bound} \cite[Proposition 2.6.1]{Finschi-thesis} to give a slightly stronger bound.  
Note that as a consequence of Lemma \ref{perturbation-lemma}, our bound is applicable to all oriented matroids 
rather than just uniform oriented matroids.


\begin{theorem}\label{thm:Finschi-improved}
Let $\mcM$ be an oriented matroid of rank $r$ on $n$ elements, and let $X,Y \in \mcC^*(\mcM)$ with $X \neq -Y$. Then 
\begin{equation} \label{quadbound-part1}
d_{\mcM}(X,Y) \leq n-r+1 + \sum_{k=2}^{|X^0 \setminus Y^0|-1}\left(\left \lfloor \frac{n-r-k}{2} \right\rfloor + 1\right).
\end{equation}
In particular, when $r \geq 4$ and $n-r \geq 2$,
\begin{equation} \label{quadbound-part2}
\diam(G^*(\mcM)) \leq n-r+1 + \sum_{k=2}^{\min(r-2,n-r)} \left(\left\lfloor \frac{n-r-k}{2}\right\rfloor +1 \right).
\end{equation}
\end{theorem}

This bound contrasts the best-known upper bounds on polytope diameters, which are linear in fixed dimension, but grow exponentially in the dimension (e.g., \cite{Kalai-Kleitman} and  \cite{Eisenbrandetal2010}). For a survey of the best bounds and more updates about diameters of polytopes see \cite{Criado2017,CriadoSantos, Eisenbrandetal2010, Santos-recent, Sukegawa19} and the references therein.

It is not immediately clear whether bounds on the diameter of the cocircuit graph of a realizable oriented matroid imply bounds on polytope diameters. This possible connection has been discussed before. For example, a connection of the (original) Hirsch conjecture to Conjecture \ref{conj:mainproblem} was stated in Remark 4.3 of \cite{Felsner-et-al}. Here we clarify how a quadratic bound for oriented matroids may  have implications for the polynomial Hirsch conjecture of polytopes.  For this it is important to ask a related question: if $X$ and $Y$ are vertices in a (poly)tope $\mcT$, does the shortest path from $X$ to $Y$ in the supergraph $G^*(\mcM)$ leave the region/tope $\mcT$? 
If the following conjecture is true, it would imply a quadratic bound on the diameter of polytopes! We have checked the validity of Conjecture \ref{London-Paris-Conj} for oriented matroids up to nine elements with computers.

\begin{conjecture}\label{London-Paris-Conj}
Let $\mcM$ be a uniform oriented matroid, and let $X,Y \in \mcC^*(\mcM)$ be cocircuits that are vertices of at least one tope of $\mcM$. Then there exists a tope $\mcT$ in $\mcM$ such that $X, Y \in \mcT$ and $d_{\mcM}(X, Y) = d_{\mcT}(X, Y)$.
\end{conjecture}

The rest of the paper is structured as follows.  In Section \ref{section:overview}, we quickly review the key aspects of oriented matroids that will be relevant for us. In Section \ref{section:reductions}, we prove  Lemma \ref{perturbation-lemma}, reducing Conjecture \ref{conj:mainproblem} to studying uniform oriented matroids.  We also establish some simple lower bounds on diameter and show the bound in Conjecture \ref{conj:mainproblem} cannot be improved because the distance between antipodal cocircuits is exactly $n-r+2$  (see the proof of Lemma \ref{lem:distance-lower-bound}). Section \ref{section:smallmatroids} begins with computational results that establish Conjecture \ref{conj:mainproblem} for uniform oriented matroids whose ground set has at most nine elements (see Theorem \ref{thm:computational-result}).  We then move on to establish Conjecture \ref{conj:mainproblem} in low rank and low corank, including a nice geometric argument for uniform oriented matroids in rank three (see Theorem \ref{thm:Hirsch-UOM-rank3}).  Section \ref{section:quadraticbound} discusses a stronger quadratic upper bound on the diameter of a uniform oriented matroid (see Theorem \ref{thm:Finschi-improved}). We conclude with a discussion of Conjecture \ref{London-Paris-Conj}.


\section{A Quick Review of Oriented Matroids} \label{section:overview}

Let $E = \{\vv_1,\ldots,\vv_n\} \subseteq \mathbb{R}^r$ be any set of vectors.  For simplicity, we will assume $E$ spans $\mathbb{R}^r$.  We will not make a distinction between $E$ as a set of vectors or $E$ as a matrix in $\bbR^{r \times n}$.  In classical matroid theory, we consider the set of linear dependences among the vectors in $E$.  In oriented matroid theory, we consider not only the set of linear dependences on $E$, but also the signs of the coefficients that make up these dependences.  To any linear dependence $\sum_{i=1}^n z_i\vv_i = \zerovec$ we associate a \textit{signed vector} $\left(\sign(z_i)\right)_{i=1}^n$.  The \textit{sign} of a number $z \in \mathbb{R}$, denoted $\sign(z) \in \{+,-,0\}$, encodes whether $z$ is positive, negative, or equal to $0$. If $\vz = (z_1,\ldots,z_n) \in \mathbb{R}^n$ is a vector, we use $\sign(\mathbf{z})$ to denote the vector of signs: $\sign(\mathbf{z}) := \left(\sign(z_i)\right)_{i=1}^n \in \{+,0,-\}^n.$  We define the set of \textit{signed vectors} on $E$ as 
$$
\mcV(E) = \{\sign(\vz): \vz \text{ is a linear dependence on } E\}.
$$  In other words, $\mcV(E) = \{\sign(\vz): E\vz = \zerovec\}$.  

Among all signed vectors determined by linear dependences on $E$, those with minimal (and nonempty) support under inclusion, are called the \textit{signed circuits} of $E$.  The set of such signed circuits is denoted $\mcC(E)$.  

Dually, for any $\vc \in \bbR^r$, we can consider the \textit{signed covector} $\left(\sign(\vc^T\vv_i)\right)_{i=1}^n$.  The set of 
all signed covectors on $E$ is 
$$
\mcV^*(E) = \{\sign\left(\vc^TE\right):\vc \in \bbR^r\}.
$$
The set of signed covectors of minimal, nonempty support are called \textit{signed cocircuits} and are denoted by $\mcC^*(E)$. It is important to note that if $X$ is a cocircuit, then so is $-X$.

Summarily, to any collection of vectors $E \subseteq \mathbb{R}^r$, there are four sets of  vectors that encode dependences among $E$.  Those are the signed vectors $\mcV(E)$ arising from linear dependences, the signed circuits $\mcC(E)$ arising from minimal linear dependences, the signed covectors $\mcV^*(E)$ arising from valuations of linear functions, and signed cocircuits $\mcC^*(E)$ arising from linear valuations of minimal support.  The first fundamental result in oriented matroid theory shows that any one of these sets is sufficient to determine the other three \cite[Corollary 6.9]{Ziegler-book}.  Any oriented matroid that arises from a collection of signed cocircuits in this way is called a \textit{realizable} oriented matroid.

Now we are ready to motivate the definition of oriented matroids through a geometric model that proves to be more useful than the axiomatic definition.  Let $E = \{\vv_1, \ldots, \vv_n\} \subseteq \bbR^r$ be a collection of vectors, and let $\mcM(E)$ be the oriented matroid determined by $E$.  To each vector $\vv_i$, there is an associated hyperplane $H_i:= \{\vx \in \bbR^r : \vx^T \vv_i = 0\}$. Each $H_i$ is naturally oriented by taking $H_i^+:= \{\vx \in \bbR^r: \vx^T \vv_i > 0\}$ and defining $H_i^-$ analogously.  

Therefore, the vectors in $E$ determine a central hyperplane arrangement $\mcH$ in $\mathbb{R}^r$.  Any vector $\vx \in \mathbb{R}^r$ has an associated sign vector determined by its position relative to the hyperplanes in $\mcH$. These signs can be computed as $\sign(\vx^T\vv_i)$ for each $i$; in other words, by computing $\sign\left(\vx^TE\right)$. Therefore, the signed covectors of $\mcM(E)$ are in bijection with the regions of the hyperplane arrangement $\mcH$. 

Further, because $\sign\left(\vx^TE\right) = \sign\left((c\vx)^TE\right)$ for any positive scalar $c$, no information from $\mcH$ is lost if we intersect $\mcH$ with the unit sphere $\mathbb{S}^{r-1}$, giving a collection of codimension-one spheres $\{s_i = H_i \cap \bbS^{r-1} : H_i \in \mcH\}$.  This induces a cell decomposition of $\mathbb{S}^{r-1}$ whose nonempty faces correspond to covectors of $\mcM(E)$ and whose vertices  correspond to cocircuits of $\mcM(E)$. The regions corresponding to covectors of maximal support are called \textit{topes}. An example is illustrated in Figure \ref{figure:fig1}.  In that figure, the cocircuit $X$ is encoded by the sign vector $(+,+,0,-,0)$. Similarly, the shaded region (a tope) corresponds to the covector $(+,+,+,-,+)$.  

\begin{figure}
\centering
\includegraphics[scale=.75]{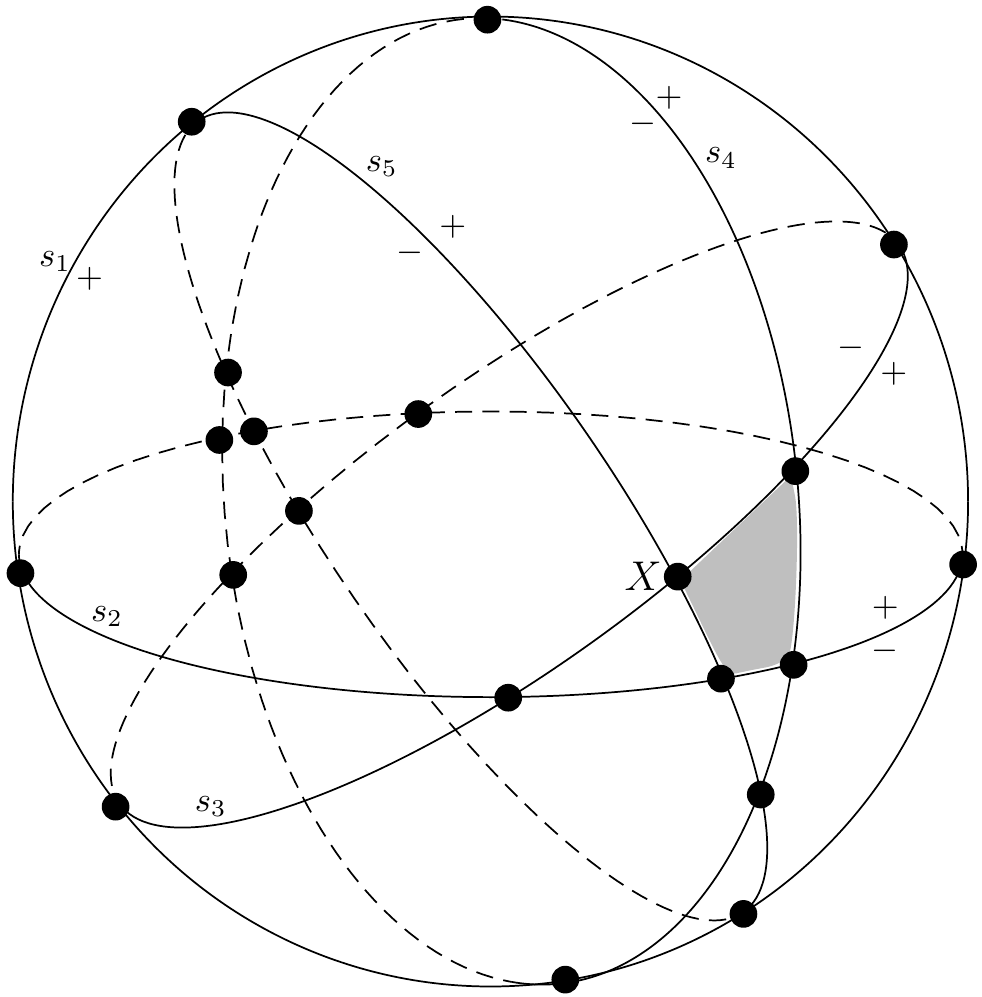}
\caption{An oriented matroid arising from an arrangement of five hyperplanes.}
\label{figure:fig1}
\end{figure}

Not all matroids can be oriented. Determining whether a matroid is orientable is an NP-complete problem, 
even for fixed rank (see \cite{OrientNP}). But, a topological model provides the ``right'' intuition for visualizing arbitrary oriented matroids.  Every oriented matroid can be viewed as an arrangement of equators on a sphere, as in the realizable case, provided that one is allowed to slightly perturb the spheres determined by $H_i \cap \mathbb{S}^{r-1}$ in the following way. 

Let $Q$ be an equator of $\bbS^{r-1}$; that is, the intersection of $\bbS^{r-1}$ with some $(r-1)$-dimensional subspace of $\bbR^r$. If $\varphi: \bbS^{r-1} \rightarrow \bbS^{r-1}$ is a homeomorphism, then the image of the equator $\varphi(Q) \subseteq \bbS^{r-1}$ is called a \textit{pseudosphere}. Because $Q$ decomposes $\bbS^{r-1}$ into two pieces, so too does $\varphi(Q)$.  Therefore, we may define an 
\textit{oriented pseudosphere} to be a pseudosphere, $s$, together with a choice of a positive side $s^+$ and negative side $s^-$.  
Now we may define an \textit{arrangement} of pseudospheres in $\bbS^{r-1}$ to be a finite collection of pseudospheres $\mcP = \{s_e: e \in E\} \subseteq \bbS^{r-1}$ such that 

\begin{enumerate}
\item for any subset $A \subseteq E$, the set $S_A = \bigcap_{e \in A}s_e$ is a topological sphere, and
\item if $S_A \not\subseteq s_e$ for $A \subseteq E$ and $e \in E$, then $S_A \cap s_e$ is a pseudosphere in $S_A$ with two parts, 
$S_A \cap s_e^+$ and $S_A \cap s_e^-$. 
\end{enumerate}

A pseudosphere arrangement is \textit{essential} if $\bigcap_{e \in E}s_e = \emptyset$.  Any essential pseudosphere arrangement $\mcP$ induces a regular cell decomposition on $\bbS^{r-1}$.  Because each pseudosphere in $\mcP$ has a positive and negative side, the cells of this decomposition are naturally indexed by sign vectors in $\{+,-,0\}^E$.  We use $\Gamma(\mcP)$ to denote the poset of such sign vectors, ordered by face containment.  We have encountered this same (abstract) poset before as $\Gamma(\mcL)$ in the introduction, the poset induced over the set of covectors $\mcL$ of an oriented matroid. 
As it turns out the following theorem of Folkman and Lawrence gives an exact correspondence between oriented matroids and pseudosphere arrangements. The same sets of sign vectors appear in both cases.

\begin{theorem} {\rm{(Topological Representation Theorem \cite{Folkman-Lawrence}})}

Let $\mcP$ be an essential arrangement of pseudospheres in $\bbS^{r-1}$.  Then $\Gamma(\mcP) \cup \{\zerovec\}$ is the set of covectors of an oriented matroid of rank $r$.  Conversely, if $\mcV^*$ is the set of covectors of a loopless oriented matroid of rank $r$, then there exists an essential arrangement of pseudospheres $\mcP$ on $\bbS^{r-1}$ with $\Gamma(\mcP) = \mcV^* \setminus \{\zerovec\}$.
\end{theorem}

If $\mcM$ is an oriented matroid, the pseudosphere arrangement $\mcP$ guaranteed by the Topological Representation Theorem is called the \textit{Folkman-Lawrence representation} of $\mcM$. Two elements $e, f \in E$ are \emph{parallel} if $X_e = X_f$ for all $X \in \mcL$ or $X_e = -X_f$ for all $X \in \mcL$. Note that we can eliminate parallel elements without changing the pseudosphere arrangement $\mcP$.

\begin{remark}
Let $\mcM$ be a uniform oriented matroid of rank $r$.  If $A \subseteq E(\mcM)$ is any set with $|A| \leq r-1$, then $S_A = \bigcap_{e \in A} s_e$ is an $(r-1-|A|)$-dimensional pseudosphere in the Folkman-Lawrence representation $\mcP(\mcM)$. 
\end{remark}

Let $\mcM$ be an oriented matroid of rank $r$, and let $\mcP$ be the Folkman-Lawrence representation  of $\mcM$.  Then the underlying graph of $\mcP$ (as a cell complex) is the cocircuit graph $G^*(\mcM)$.  This provides a geometric model for visualizing cocircuit graphs of oriented matroids. A \emph{coline} in $\mcM$ is a one-dimensional sphere in the Folkman-Lawrence representation of $\mcM$. In matroidal language, a coline is a covector that covers a cocircuit in the natural component-wise partial order where $0 < +,-$.  For a uniform oriented matroid of rank $r$, a coline is a covector $U$ with $|U^0| = r-2$. Further, in a uniform oriented matroid, for each subset $S \in {[n] \choose r-2}$, there exists a coline $U$ with $U^0 = S$. The graph of any coline is a simple cycle of length $2(n-r+1)$.

The Folkman-Lawrence representation gives us a more concrete topological understanding of the following operations on oriented matroids. Let $\mcM$ be an oriented matroid on ground set $E$ with signed covectors $\mcV^*(\mcM)$, and let $A\subseteq E$. The \emph{restriction} of a sign vector $X \in \{+, -, 0\}^E$ to $A$ is the sign vector $X|_A\in \{+, -, 0\}^A$ defined by $(X|_A)_e = X_e$ for all $e \in A$.  The \emph{deletion} $\mcM \backslash A$ is the oriented matroid with covectors $$\mcV^*(\mcM \backslash A)= \{X|_{E\backslash A} \ : \ X \in \mcV^*(\mcM) \}\subseteq \{+, -, 0\}^{E\backslash A}.$$ The \emph{contraction} $\mcM / A$ is the oriented matroid with covectors $$\mcV^*(\mcM/A) = \{X|_{E\backslash A} \ : \ X \in \mcV^*(\mcM), A\subseteq X^0 \}\subseteq \{+, -, 0\}^{E\backslash A}.$$ The fact that $\mcM \backslash A$ and $\mcM/A$ are oriented matroids is proved in \cite[Lemma 4.1.8]{OMBook}.

 The deletion $\mcM \backslash A$ is also referred to as the restriction of $\mcM$ to $E \backslash A$. Geometrically, $\mcM \backslash A$ is the oriented matroid of the same rank as $\mcM$ obtained by removing pseudospheres $\{s_e \ : \ e \in A \}$. The contraction $\mcM / A$ is the oriented matroid of obtained by intersection $S_A$ with $\{s_e \ : \ e \in E \backslash A\}$.
 
Note also that the pseudosphere arrangement of an oriented matroid of rank  $r$ lies on the sphere $\bbS^{r-1}$. The \emph{topes} correspond to the regions, homeomorphic to balls of dimension $r-1$, that partition the sphere. For realizable oriented matroids coming from a hyperplane arrangement, topes are actual convex polytopes. 

Given a tope $\mcT$ of an oriented matroid $\mcM$, we define  its graph as the subgraph of
$G^*(\mcM)$ induced by the cociruits of $\mcM$ in $\mcT$. Next, we show the graph of a tope $\mcT$ in a uniform oriented 
matroid $\mcM$ of rank $r$ on $n$ elements, is isomorphic to a graph of an abstract polytope of dimension 
$r-1$ on $n$ elements. Abstract polytopes, an abstraction of simple polytopes,
were introduced by Adler and Dantzig \cite{ad} for the purpose of studying the diameter of their graphs. Abstract polytopes have been further generalized in recent years by several authors (see \cite{Eisenbrandetal2010, Santos-recent} and references there 
for details).

   \begin{definition}	\label{def:abstract_polytope}
	Let $T$ be a finite set. A family $\mcA$ of subsets of $T$ (called
	{\it vertices}) forms a {\it d-dimensional abstract polytope on the ground set $T$}  if the following three axioms
	are satisfied:
	\begin{enumerate}
		\item[(i)]
		Every vertex of $ \mcA$ has cardinality $d$.
		\item[(ii)] Any subset of $d - 1$ elements of $T $ is either contained in no vertices
		of $ \mcA$ or in exactly two (called neighbors or adjacent vertices).
		\item[(iii)] Given any pair of distinct vertices $X,Y \in  \mcA$, there exists a sequence of vertices
		\newline \noindent
		$X=Z_0, Z_1, \ldots, Z_k=Y$ in $\mcA$ such that
		\begin{enumerate}
			\item[(a)]
			$Z_i,Z_{i+1}$ are adjacent for all $i=0,1, \ldots, k-1$, and
			\item[(b)]
			$X \cap Y \subset Z_i$ for all  $ i=0,1, \ldots, k.$  
		\end{enumerate}
	\end{enumerate}

The graph $G_{abs}(\mcA)$ of an abstract polytope $\mcA$ is composed of nodes corresponding to its vertices, where two vertices 
are adjacent on the graph as specified in axiom (ii).
\end{definition} 

\medskip

Consider a simple polytope $\mcP$ of dimension $d$ which is the intersection of $n$ facet-defining half-spaces. Then, indexing the $n$ facets
by $1, \ldots, n$, the family of all sets of indices that define a vertex of $\mcP$ is an abstract polytope of dimension $d$ on the ground set $\{1, \ldots, n\}$.
In particular, the three axioms of abstract polytopes state that the graph  $G(\mcP)$ associated with the vertices of $\mcP$ has the following 
three properties:
	\begin{enumerate}
	\item[(i)]
    $G(\mcP)$ is regular of degree $d$ (as all the hyperplanes corresponding to the half-spaces are in general position.)
	\item[(ii)]
	All  edges of $G(\mcP)$ have two vertices as end points (as $\mcP$  is bounded).
	 	\item[(iii)]
	 	For any two vertices $X,Y$ that lie in a face $F$ of $\mcP$, there exists a path between the nodes corresponding 
		to $X$ and $Y$ on $G(\mcP)$  composed entirely of nodes corresponding to vertices on $F$ (as $F$ is also a polytope.)
	  \end{enumerate}
Interestingly, while the axioms of abstract polytopes represent only three  basic properties  related to graphs of simple polytopes, a substantial number
 of the results related to diameter of simple polytopes in \cite{Klee-Walkup} have been proved in \cite{ad} for abstract polytopes. 
 
 Next, we show that these properties are satisfied by the graph of topes of uniform oriented matroids.

\begin{lemma}\label{lem:tope_AP_equiv}
	Given a uniform oriented matroid $\mcM = (E,\mcC^*)$ of rank $r\geq 2$ and a tope $\mcT$ of $\mcM$,
	let 
	\begin{equation*} \label{eq:tope_ap}
	\mcC_{\mcT}=\{X \in  \mcC^*: X  < \mcT  \},  
	\mbox{ and }
	\mcA=\{X^0 : \  X \in 	\mcC_{\mcT} \} .
	\end{equation*}
	Then, $\mcA$ is a $d$-dimensional  abstract polytope on the ground set $E$, where $d=r-1$. Moreover, the graph $ G(\mcT)$ of $\mcT$ is isomorphic to the graph $G_{abs}(\mcA)$ of $\mcA$.
	\end{lemma}
	
	\begin{proof} 
	
		We show that  $ \mcA$ satisfies the three axioms of abstract polytopes: 
		\begin{enumerate}
			\item[(i)] Axiom (i) holds because  $\mcM$ is a uniform oriented matroid of rank $r$. 
			\item[(ii)] 	Let $E' \subset E$ such that $|E'|=d-1$, 
			and assume
			that there exists $X \in\mcC_{\mcT}$  such that
		 $E' \subset X^0$ (otherwise, no vertex of $\mcA$ contains $E'$ and we are done). 
			  Let $U=\{ W \in \mcM^* : E' \subset W^0\}$, then $U$ 
			  is a coline of $\mcM$ whose graph is a simple cycle. 
			 Let $Y_1,Y_2$ be the two adjacent cocircuits to $X$ in $U$. Then, there exists an element 
			 $e \in E \setminus E'$
			 such that $S(Y_1,Y_2)=e$ and
			 $S(X, Y_i)=\emptyset \;(i=1,2)$, implying that exactly one of $Y_1,Y_2$,  say $Y_1$, is in $\mcT$. However, no other cocircuit in $U$ is in $\mcT$.
			Suppose, to the contrary, that there exists $Z \in U$, distinct from $X$ and $Y_1$ that belongs to $\mcT$. 
			Then by definition 
			$$|X^0 \cap Y^0_1| = |X^0 \cap Z^0| = |Y^0_1 \cap Z^0| = d - 1, \text{ and } S(X, Y_1) = S(X, Z) = S(Y_1, Z) = \emptyset.$$
			This means that $X$,  $Y_1$, and $Z$, are all adjacent on $U$.  As 
			the graph of $U$ is a simple cycle of size $2(n-r+1)$,	this leads to contradiction.			
				\item[(iii)] 
				By \cite[Theorem 2.3]{Felsner-et-al}, for any 
				$X,Y \in \mcC^*$  there exists an $(X,Y) $  crabbed path on $G^*(\mcM)$. That is, there exists a path 
				$X=Z_0, Z_1, \ldots, Z_k=Y$ on $G^*(\mcM)$ such that $Z_i^+  \subseteq X^+ \cup Y^+$ and $Z_i^-  \subseteq X^- \cup Y^-$ for 
				all $0 \leq i \leq k$. This implies that 
				if $X,Y \in \mcC_{\mcT}$, then  for $i=1, \ldots, k-1$,
			  $ Z_i \in  \mcC_{\mcT}$ (as $Z_i <  \mcT $), so $Z^0_i \in \mcA$, and
			       $X^0 \cap Y^0 \subseteq  Z_i ^0$.
				Now, let $G(\mcT)$ be the
				graph of  $\mcT$. 
				Note that as $S(X,Y)=\emptyset$ for any 
			      $X,Y \in \mcC_{\mcT}$, $X$ and $Y$ share an edge on $G(\mcT)$ if and only if $|X^0 \cap Y^0|=d-1$.
				 However, the two vertices on $G_{abs}(\mcA)$ corresponding 
				to $X^0,Y^0$ are adjacent if and only if $|X^0 \cap Y^0|=d-1$. Thus, we conclude that 
				$G(\mcT)$ is isomorphic to $G_{abs}(\mcA)$, so Axiom (iii) is satisfied.	
				
						\end{enumerate}
				Note that by the proof of part (iii) above we have that the graph $ G(\mcT)$ of $\mcT$ is isomorphic to the graph $G_{abs}(\mcA)$ of $\mcA$.	
	                           \end{proof}
\section{Reductions and Lower Bounds}
\label{section:reductions}

For the ease of notation, let $OM(n,r)$ be the set of all oriented matroids of rank $r$ whose ground set has cardinality $n$.  Let $UOM(n,r)$ be the set of all uniform oriented matroids in $OM(n,r)$. Let $\Delta(n,r)$ denote the maximal diameter of $G^*(\mcM)$ as $\mcM$ ranges over $OM(n,r)$. Klee and Walkup \cite{Klee-Walkup} showed that the maximal diameter among all $d$-dimensional polytopes with $n$ facets is achieved by a simple polytope. Their argument was straightforward: if $P$ is a $d$-polytope with $n$ facets that is not simple, then slightly perturbing the facets of $P$ will produce a simple polytope whose diameter is at least as large as that of $P$. Our goal in this section is to prove an analogous result for oriented matroids. First we require some definitions, see \cite[Section 7.1 and 7.2]{OMBook} for more details.



Let $\mcM$ be an oriented matroid on ground set $E$. An \emph{extension} of $\mcM$ is an oriented matroid $\widetilde{\mcM}$ on a ground set $\widetilde{E}$ that contains $E$, such that the restriction of $\widetilde{\mcM}$ to $E$ is $\mcM$. We say $\widetilde{\mcM}$ is a \emph{single element extension} if $|\widetilde{E} \backslash E| = 1$. For any single element extension $\widetilde{\mcM}$, there is a unique way to extend cocircuits of $\mcM$ to cocircuits of $\widetilde{\mcM}$. Specifically, there is a function $$\sigma: \mcC^*(\mcM) \to \{+, -, 0\}$$ such that $\sigma(-Y) = -\sigma(Y)$ for all $Y \in \mcC^*(\mcM)$ and $$\{(Y, \sigma(Y)): \ Y \in \mcC^*(\mcM)\} \subseteq \mcC^*(\widetilde{\mcM}).$$ That is, $(Y, \sigma(Y))$ is a cocircuit of $\widetilde{\mcM}$ for every cocircuit $Y$ of $\mcM$. The functions $\sigma: \mcC^* \to \{+, -, 0\}$ that correspond to single element extensions are called \emph{localizations}. Furthermore, $\widetilde{\mcM}$ is uniquely determined by $\sigma$, with \begin{multline*}\mcC^*(\widetilde{\mcM}) = \{(Y, \sigma(Y)):\ Y \in \mcC^*(\mcM)\} \cup \\ \{(Y^1 \circ Y^2, 0):\ Y^1, Y^2 \in \mcC^*(\mcM), \sigma(Y^1) = -\sigma(Y^2) \neq 0, S(Y^1, Y^2) = \emptyset, \rho(Y^1 \circ Y^2) = 2\}. \end{multline*} Here $\rho$ is the rank function and $\circ$ is the composition of covectors.


Now we are ready to define the perturbation map on non-uniform oriented matroids. 

\begin{definition}\cite[Theorem 7.3.1]{OMBook} \label{PerturbationMap}
   Let $\mcM$ be an oriented matroid of rank $r \geq 2$ on $E$. If $f \in E$ is not a coloop, then $\mcM$ is a single element extension of a rank $r$ oriented matroid $\mcM_0 := \mcM \backslash f$, with localization $\sigma_f$. Let $\overline{W} \in \mcC^*(\mcM_0)$ be a cocircuit with $\sigma_f(\overline{W}) = 0$, meaning $W = (\overline{W}, 0)$ is a cocircuit of $\mcM$. Then the local perturbation $\mcM'$ of $\mcM$ can be defined as a single element extension of $\mcM_0$ with localization $$
\sigma_{LP}(\overline{Y}) = 
\begin{cases}
+ & \text{ if } \overline{Y} = \overline{W}, \\
- & \text{ if } \overline{Y} = -\overline{W}, \\
\sigma_f(\overline{Y}) & \text{otherwise}.
\end{cases}
$$ 
\end{definition}

We can now reduce the general diameter problem to the case of uniform oriented matroids, as promised by Lemma \ref{perturbation-lemma}.


\begin{proof} (of Lemma \ref{perturbation-lemma})

 Let $\mcM$ be a non-uniform oriented matroid. We may assume without loss of generality that, $\mcM$ does not have any loops, coloops or parallel elements since removing them will not affect the cocircuit graph of $\mcM$. Note that there exists $W \in \mcC^*(\mcM)$ with $|W^0| > r-1$. Pick an arbitrary $f \in W^0$. Let $\mcM_0 := \mcM \backslash f$ and let $\mcM'$ be the perturbed oriented matroid defined in Definition \ref{PerturbationMap}.  We will show $\diam(\mcM) \leq \diam(\mcM')$. In addition, if $\mcM$ is realizable, then we will show the perturbed $\mcM'$ can also be made realizable. From this, it will follow that for all $n$ and $r$, the optimal bound $\Delta(n, r)$ is achieved by a uniform oriented matroid.

\begin{figure}[h]
\centering
\begin{minipage}{.3\textwidth}
\includegraphics[width=.95\textwidth]{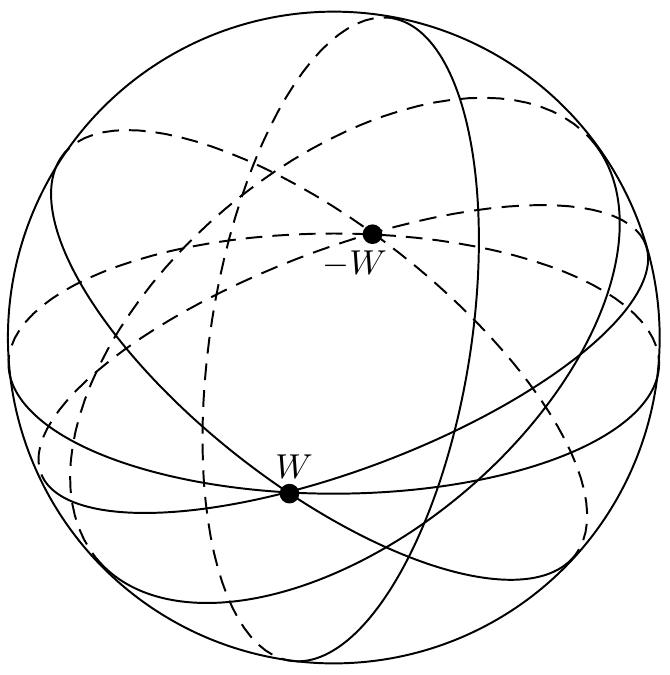}
\end{minipage}
\begin{minipage}{.3\textwidth}
\includegraphics[width=.95\textwidth]{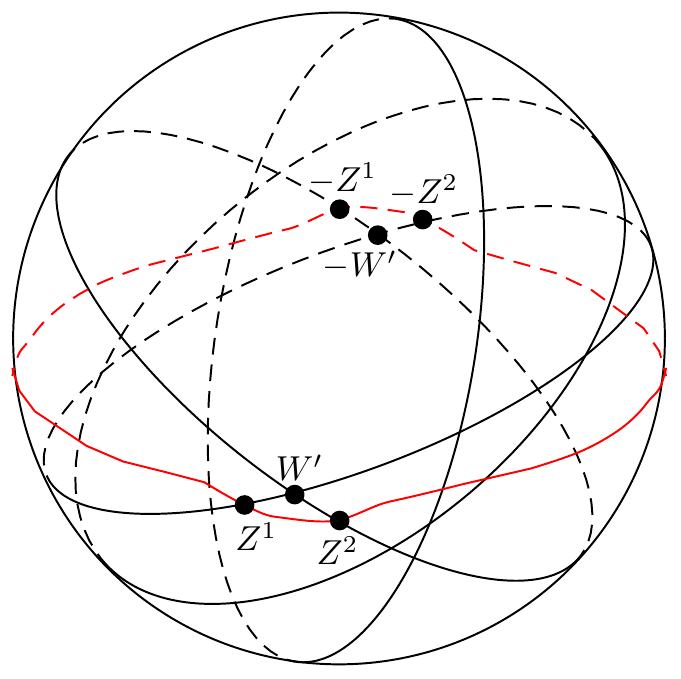}
\end{minipage}
\begin{minipage}{.3\textwidth}
\includegraphics[width=.95\textwidth]{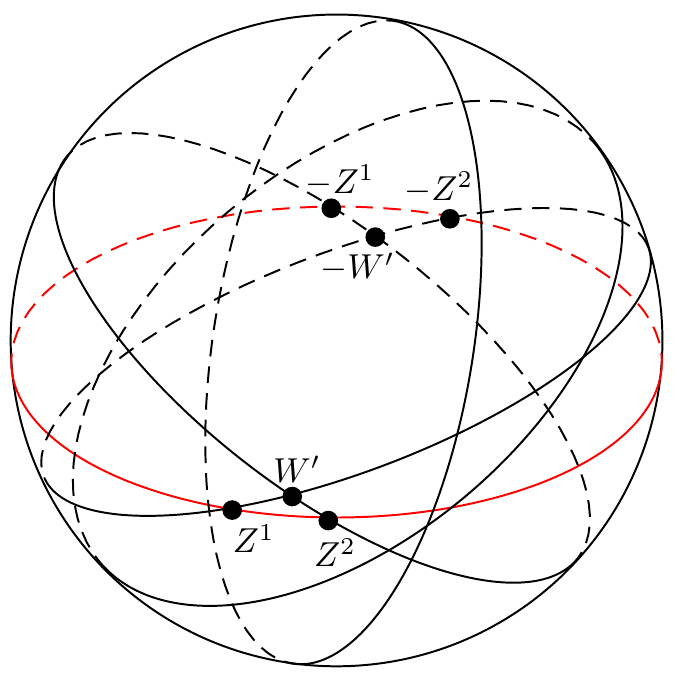}
\end{minipage}
\caption{A non-uniform oriented matroid (left), a local perturbation (center), and a realizable local perturbation (right).}
\label{fig:perturbation-figure}
\end{figure}

Denote by $\{X^1, X^2, \ldots, X^k\} = \{X \in \mcC^*(\mcM_0): \sigma_f(X) = -, \ S(\overline{W}, X) = \emptyset, \ \rho(\overline{W}, X) = 2\}$. Note that $X^1, \ldots, X^k$ are exactly the cocircuits that are adjacent to $\overline{W}$ in $G^*(\mcM_0)$  before the extension with $\sigma_f(X^i) = -$. Let $Z^i  = (X^i\circ \overline{W}, 0)$. After the perturbation by $\sigma_{LP}$, $W$ is mapped to $W'= (\overline{W}, +)$. Since $\sigma_{LP}$ and $\sigma_f$ only differ on $\pm \overline{W}$, it follows that $\pm Z^1, \ldots, \pm Z^k$ are all the cocircuits created by this perturbation. After the perturbation, each edge of the form $\{W,X^i\}$ in $G^*(\mcM)$ is subdivided into two edges $\{W’,Z^i\}$ and $\{Z^i,X^i\}$ (similarly $\{-W, -X^i\}$ is subdivided into $\{-W', Z^i\}$ and $\{-Z^i, -X^i\}$).

Now let $X,Y \in \mcC^*(\mcM)$ be any two cocircuits of $\mcM$ such that $X, Y \in \mcC^*(\mcM')$ ($X, Y$ could be $\pm W$, in this case we just consider $\pm W'$ in $\mcM'$). Take a minimal path between $X$ and $Y$ on $G^*(\mcM')$, and replace any elements of $\{ \pm W', \pm Z^1, \ldots, \pm Z^k\}$ with $\pm W$ respectively. This gives us a path (potentially having repeated elements and not necessarily shortest) between $X$ and $Y$ in $\mcM$. Now if we pick $X, Y \in \mcC^*(\mcM)$ that realize the diameter of $\mcM$, since $d_{\mcM}(X, Y) \leq d_{\mcM'}(X, Y)$, we have $\diam(\mcM) = d_{\mcM}(X, Y) \leq d_{\mcM'}(X, Y) \leq \diam(\mcM')$.


Now suppose $\mcM$ is realizable. Let $\mcH = \{H_1, \ldots, H_n\}$ be the hyperplane arrangement corresponding to $\mcM$ (with $f$ corresponding to $H_n$). Let $H_i = \{\vx: \vx^T\vv_i = 0\}$, and $\vw$ be the vector realizing $W$. Note that we have $\vw^T\vv_n = 0$ since the last entry of $W$ is 0.  Consider $\vy$,  the minimizer of $\vx^T\vv_n$ over all cocircuits of $\mcM$ subject to $\vx^T\vv_n>0$. Now we replace $H_n$ by $H_n'=\{\vx: \vx^T((1-\epsilon)\vv_n + \epsilon \vy) = 0\}$, in which the choice of $\epsilon$ will be made later. Note that, $$\vx^T((1-\epsilon)\vv_n + \epsilon \vy) = \vx^T\vv_n - \epsilon \vx^T\vv_n + \epsilon \vx^T \vy.$$ We first pick the sign of $\epsilon$ so that $\epsilon \vw^T \vy > 0$; as a result, $\vw \in H_n'^+$ and $-\vw \in H_n'^-$. Then we take $|\epsilon|$ small enough such that $|\vx^T\vv_n| > |\epsilon(\vx^T\vv_n - \vx^T\vw')|$ for all $\vx$ vectors that realize a cocircuit in $\mcM$ (this choice of $\epsilon$ exists since the number of cocircuits is finite and we may scale the vector). The construction ensures that all cocircuits, except those that lie on $H_n$ with degeneracy, will have the same sign as defined in Definition \ref{PerturbationMap}. As a result $\mcH' = \{H_1, \ldots, H_{n-1}, H_n'\}$ corresponds to some realizable oriented matroid $\mcM'$ after some local perturbations (the composition of perturbation maps on all cocircuits with degeneracy on $H_n$ (including $W$) as defined in Definition \ref{PerturbationMap}). 
 
To conclude, we have decreased the number of pairs of $(W, f)$  with $|W^0|> r-1$ and $W_f = 0$ without decreasing the diameter. By continuing this procedure, we will eventually obtain an oriented matroid in which no such pair of $(W, f)$ can be found, or equivalently $|X^0| = r - 1$ for all $X \in \mcC^*(\mcM)$. Hence $\Delta(n, r)$ will be achieved by a uniform oriented matroid.
\end{proof}

Hence it suffices to study uniform oriented matroids for the purpose of bounding $\Delta(n,r)$.  The bound in Conjecture \ref{conj:mainproblem} can be rewritten as $\Delta(n,r) \leq n-(r-1)+1$.  For polytopes, $n-(r-1)+1 = n-d+1$.  It may seem mysterious that the bound here is one more than the Hirsch bound, so we will pause for a moment to discuss this.  We begin by proving Lemma \ref{lem:distance-lower-bound} from the Introduction.


\begin{proof} (of Lemma \ref{lem:distance-lower-bound})

Recall that if cocircuits $Z$ and $W$ are adjacent in $G^*(\mcM)$, then there are elements $e \in Z^0 \setminus W^0$ and $e' \in W^0 \setminus Z^0$ such that $Z^0 = (W^0 \setminus \{e'\})\cup \{e\}.$ In other words, when we move from $Z$ to $W$, we see $Z_e = 0$ change to become $W_e \neq 0$ and $Z_{e'} \neq 0$ change to become $W_{e'} = 0$.  Therefore, we will say that each edge in $G^*(\mcM)$ encodes two ``basic transformations'', which are changes to the cocircuit that transform a nonzero entry into a zero entry or vice versa.

Now we consider the differences in the sign patterns of $X$ and $Y$.  For each $e \in S(X,Y)$ we require two basic transformations to move from $X$ to $Y$: one to transform $X_e$ to $0$, and another to transform $0$ to $-X_e = Y_e$.  For each $e \in X^0 \setminus Y^0$, we require one basic transformation to transform $0$ to $Y_e$.  Similarly, for each $e \in Y^0 \setminus X^0$, we require one basic transformation to transform $X_e$ to $0$.  Therefore, moving from $X$ to $Y$ requires at least $2|S(X,Y)| + |X^0 \setminus Y^0| + |Y^0 \setminus X^0| = 2|S(X,Y)| + 2|X^0\setminus Y^0|$ basic transformations.  Thus $d_{\mcM}(X,Y) \geq |S(X,Y)| + |X^0 \setminus Y^0|$.

Now we examine the case where $X=-Y$ more closely.  In this case, $S(X,Y) = \supp(X)$ and $X^0 = Y^0$.  Pick a shortest path from $X$ to $Y$ in $G^*(\mcM)$ and let $Z$ be the neighbor of $X$ on this path. Then $|S(Y,Z)| = n-r$ and $|Z^0 \setminus Y^0| = 1$, so $d_{\mcM}(Y,Z) \geq n-r+1$ by the above argument.  Therefore, $d_{\mcM}(X,Y) = 1 + d_{\mcM}(Y,Z) \geq n-r+2.$



Next, consider the case $|X^0 \setminus Y^0| \leq 1$. We show that the equality holds for expression \eqref{eq:distbound}.

Let $A \subseteq X^0 \cap Y^0$ have cardinality $r-2$.  If $|X^0 \setminus Y^0| = 1$, then $A = X^0 \cap Y^0$; otherwise, $X=-Y$ and we can pick $r-2$ elements arbitrarily from $X^0 = Y^0$. Let $\{s_e : e \in E \}$ be the pseudospheres in the Folkman-Lawrence representation of $\mcM$ and let $S_A = \bigcap_{e \in A} s_e$.  Because $\mcM$ is uniform, we know $S_A \approx \bbS^1$. 

We saw above that in general $d_{\mcM}(X,Y) \geq 1+|S(X,Y)|$. On the other hand, the elements of $S(X,Y)$ are in bijective correspondence with cocircuits along the shortest path from $X$ to $Y$ in $S_A$. Indeed, if $Z$ is such a cocircuit, then $Z$ and $-Z$ are antipodal vertices on $S_A$, so they constitute a $0$-dimensional pseudosphere whose positive side contains one of $X$ or $Y$ and whose negative side contains the other. Thus the distance from $X$ to $Y$ on $S_A$ is exactly $1+|S(X,Y)|$. This proves $d_{\mcM}(X,Y) \leq 1+|S(X,Y)|$.
\end{proof}

One could hope that $d_{\mcM}(X,Y) \leq n-r+1$ provided $X,Y \in \mcC^*(\mcM)$ are not antipodal cocircuits.  However, this is not the case.  Matschke, Santos, and Weibel \cite{MSW} built on the methodology of Santos's original non-Hirsch polytope  \cite{Santos} to construct a simple polytope $P_{20,40}$ of dimension 20 with 40 facets which has diameter 21. Let $\mcM_{20,40}$ be the oriented matroid obtained by lifting $P_{20,40}$ into $\mathbb{R}^{21}$ and intersecting its hyperplane arrangement with the unit sphere. Since $P_{20,40}$ is simple, $\mcM_{20,40}$ is uniform, and one of its topes is $P_{20,40}$. We will show that the oriented matroid $\mcM_{20,40} \in UOM(40,21)$ has a pair of non-antipodal cocircuits $X$ and $Y$ such that $d_{\mcM_{20,40}}(X,Y) \geq 21 = n-r+2.$

\begin{proof}(of Proposition \ref{santos-diameter})

 Let $X, Y$ be the pair of cocircuits that are of distance 21 in $P_{20,40}$. Let $E = \{1, \ldots, 40\}$. After reorientation and relabeling, we may assume that $X^0 = \{1, 2, \ldots, 20\}$, $X^+ = \{21, \ldots, 40\}$ and $Y^0 = \{21, \ldots, 40\}$, $Y^+=\{1, \ldots, 20\}$. 
    
Consider a shortest path, $\gamma$, from $X$ to $Y$ in $\mcM_{20,40}$. If each cocircuit on $\gamma$ belongs to the tope $P_{20,40}$, then its length is $21$. So we may suppose instead that $\gamma$ contains a cocircuit $Z$ that does not belong to $P_{20,40}$.  This means $Z^-\neq \emptyset$.

Recall the notion of a ``basic transformation'' from the proof of Lemma \ref{lem:distance-lower-bound}.  Each edge in the cocircuit graph accounts for two basic transformations, which change some entry on a cocircuit from $+/-$ to $0$ or from $0$ to $+/-$.  

Let $i \in Z^-$. If $X_i = +$ and $Y_i = 0$, then walking from $X$ to $Y$ via $Z$ requires at least $20+19+3 = 42$ basic transformations.  This is because each $j \in X^0$ requires one basic transformation to become an element of $Y^+$; each $j \in X^+ \setminus \{i\}$ requires one basic transformation to become an element of $Y^0$, and $i \in X^+$ requires two basic transformations to become an element of $Z^-$ and one additional transformation to subsequently become an element of $Y^0$.  Similarly, if $X_i = 0$ and $Y_i = +$, then walking from $X$ to $Y$ via $Z$ also requires at least $42$ basic transformations.  This tells us $d_{\mcM_{20,40}}(X,Y) \geq 21 = n-r+2.$
\end{proof}
%
%
%

\section{Results for small matroids} 
\label{section:smallmatroids}

\subsection{Computer-based results for oriented matroids with few elements}\hfill

Finschi and Fukuda \cite{Fukuda-Finschi} computed the exact number of isomorphism classes of uniform oriented matroids, and gave a representative of each isomorphism class, when $n \leq 9$ and in small rank/corank when $n=10$.   We established Conjecture \ref{conj:mainproblem} for all of these examples using computers. 

 \begin{table}[!ht]
 \centering
\begin{tabular}{|c|c|c|c|c|c|c|c|c|c|}
\hline
              & $n = 2$     & $n = 3$ & $n = 4$ &  $n = 5$ & $n = 6$ & $n = 7$ & $n = 8$ & $n = 9$ & $n = 10$\\ \hline
$r = 2$  & 1 & 1 & 1 & 1 & 1 & 1 & 1 & 1 & 1  \\ \hline 
$r=3$ & \multicolumn{1}{l|}{} & 1 & 1 & 1 & 4 & 11 & 135 & 4382 & 312356 \\ \hline
$r = 4$ & \multicolumn{2}{l|}{} & 1 & 1 & 1 & 11 & 2628 & 9276595 & unknown \\ \hline
$r=5$ & \multicolumn{3}{l|}{} & 1 & 1 & 1 & 135 & 9276595 & unknown \\ \hline
$r=6$ & \multicolumn{4}{l|}{} & 1 & 1 & 1 & 4382 & unknown \\ \hline
$r=7$ & \multicolumn{5}{l|}{} & 1 & 1 & 1 & 312356 \\ \hline
$r=8$ & \multicolumn{6}{l|}{} & 1 & 1 & 1 \\ \hline
$r=9$ & \multicolumn{7}{l|}{} & 1 & 1 \\ \hline
$r=10$ & \multicolumn{8}{l|}{} & 1 \\ \hline
\end{tabular}
\caption{Number of uniform oriented matroids for $n\leq 10$. }
\label{table:smallmatroiddata}
  \end{table}

Each isomorphism class is encoded by its chirotope representation. Chirotopes, or basis orientations, are one of the equivalent axiomatic systems for oriented matroids (see \cite[Section 3]{OMBook} for more details). For a given oriented matroid on ground set $E$, the chirotope defines a mapping $\chi : E^{r} \to \{-, 0, +\}$. For a realizable oriented matroid with vector configuration $\{\vv_1, \ldots, \vv_n\}$, $$\chi(\lambda_1, \ldots, \lambda_{r}) = \sgn(\det(\vv_{\lambda_1}, \vv_{\lambda_2}, \ldots, \vv_{\lambda_{r}})).$$ The data can be found on Finschi and Fukuda's Homepage of Oriented Matroids \cite{small-OMs}. Given a chirotope map $\chi$ of an oriented matroid of rank $r$ on $E = \{1, 2, \ldots, n\}$, we can generate the cocircuits by computing the set $\mcC^*(\chi) = \{(\chi(\lambda, 1), \chi(\lambda, 2), \ldots , \chi(\lambda, n)): \lambda \in E^{r-1}\}$. Since $\mcM$ is uniform, we add an edge between $X, Y \in \mcC^*(\mcM)$ if and only if $|X^0 \cap Y^0| = r-2$ and $|S(X, Y)| = 0$. For $n=9$, $r=5$ and $n=10$, $r=7$, the chirotope maps are missing in the original dataset. However we can look at their duals ($n=9$, $r=4$ and $n=10$, $r=3$) and consider the set of circuits instead. Below is the pseudocode for computing the set of cocircuits and circuits.

\begin{algorithm}
\caption{Construct cocircuits given the chirotope map}
\hspace*{\algorithmicindent} \textbf{Input} Cardinality $n$, rank $r$ of $\mcM$ and $\chi$ the chirotope map\\
\hspace*{\algorithmicindent} \textbf{Output} A list containing all cocircuits $\mcC^*(\mcM)$
\begin{algorithmic}
\FOR{$A\subseteq [n]$ and $|A| = r - 1$}
\STATE Initialize $\vv = 0 \in \mathbb{R}^n$
\STATE Sort and vectorize $A$ to $\lambda$
 \FOR{$i= 1$ to $n$}
   \IF{ $i \not \in A$}
    \STATE $\vv[i] \leftarrow \chi(i, \lambda)$
   \ENDIF
 \ENDFOR
 \STATE Add $\pm \vv$ to the set of cocircuits
\ENDFOR
\end{algorithmic}
\end{algorithm}

\begin{algorithm}
\caption{Construct circuits given the chirotope map}
\hspace*{\algorithmicindent} \textbf{Input} Cardinality $n$, rank $r$ of $\mcM$ and $\chi$ the chirotope map\\
\hspace*{\algorithmicindent} \textbf{Output} a list containing all circuits $\mcC(\mcM)$
\begin{algorithmic}
\FOR{$A\subseteq [n]$ and $|A| = r - 1$}
\STATE Initialize $\vv = 0 \in \mathbb{R}^n$
\STATE Sort and vectorize $A$ to $\lambda$
 \FOR{$i= 1$ to  $n$}
   \IF{$i \in A$}
    \IF{$i = \min A$}
     \STATE $\vv[i] \leftarrow 1$
    \ELSE
     \STATE $\vv[i] \leftarrow -\chi(\min(A), \lambda) \times \chi(i, \lambda)$
    \ENDIF
   \ENDIF
 \ENDFOR
 \STATE Add $\pm \vv$ to the set of circuits
\ENDFOR
\end{algorithmic}
\end{algorithm}

 After finding all the cocircuits and edges, we used the Python NetworkX package \cite{NetworkX} to construct the cocircuit graph. This package has a method for computing the diameter of a graph, and also for determining the distance between any pairs of vertices. Table \ref{table:smallmatroiddata} shows the number of isomorphism classes (up to reorientation) of uniform oriented matroids of cardinality $n$ and rank $r$. We used a MacBook Pro with quad-core 2.2GHz Intel i7 processor, as well as UC Davis Math servers to construct the cocircuit graphs and compute their diameters. When $n = 9$, $r = 4, 5$ the algorithm takes the longest to terminate. On average, each instance of an oriented matroid takes about 0.36 seconds to compute, resulting in around 38.7 days to complete the checking of all oriented matroids of cardinality nine and rank four. 
 
We investigate other interesting questions such as whether the shortest path between two cocircuits on the same tope stays on the tope (see Section 5). Our code is available on Github.\footnote{\url{https://github.com/zzy1995/OrientedMatroid}} Based on our explicit computations, we derive the following theorem for small matroids, as promised in the introduction.
 
 \
 \begin{theorem}\label{thm:computational-result}
 Let $r \leq n \leq 9$ and $\mcM \in UOM(n, r)$, then $\diam(G^*(\mcM)) = n - r + 2$. Moreover, if $X,Y \in \mcC^*(\mcM)$ with $X \neq -Y$ and $n \leq 9$, then $d_{\mcM}(X, Y) \leq n-r + 1$.
 \end{theorem}
 \

\subsection{Results in low rank}\hfill

As a next step, we explore Conjecture \ref{conj:mainproblem} in low rank.  
If $\mcM \in UOM(n,2)$, then the cocircuit graph $G^*(\mcM)$ is a cycle on $2n$ vertices, so its diameter is $n = n-r+2$.  Thus Conjecture \ref{conj:mainproblem} holds trivially when $r=2$. Now we move on to study uniform oriented matroids of rank three. 

\begin{theorem} \label{thm:Hirsch-UOM-rank3}
Let $\mcM \in UOM(n, 3)$, then $\diam(G^*(\mcM)) = n -r + 2 = n-1$.
\end{theorem}

\begin{proof}
Let $\mcM \in UOM(n,3)$ and $X,Y \in \mcC^*(\mcM)$. If $X=-Y$, then $d_{\mcM}(X,Y) = n-r+2$ by Lemma \ref{lem:distance-lower-bound}.  If $|X^0 \setminus Y^0| = 1$, then $d_{\mcM}(X,Y) \leq n-r+1$ by Lemma \ref{lem:distance-lower-bound}. So we only need to consider the case that $|X^0 \setminus Y^0| \geq 2$.  But $|X^0| = |Y^0| = r-1 = 2$, so this means $X^0 \cap Y^0 = \emptyset$. 

Identify the elements of $E(\mcM)$ with $\{1,2,\ldots,n\}$.  Let $\mcP(\mcM)$ be the Folkman-Lawrence representation of $\mcM$ with pseudospheres $\{s_1,\ldots,s_n\}$. 

Without loss of generality we can assume $X^0 = \{1,2\}$ and $Y^0 = \{3,4\}$. Let $\mcM'$ denote the restriction of $\mcM$ to $\{1,2,3,4\} \subseteq E$.  The Folkman-Lawrence representation of $\mcM'$ is obtained from $\mcP(\mcM)$ by removing $s_i$ for all $i>4$. Up to relabeling and reorientation, there is only one uniform oriented matroid of rank three on four elements. We can further assume $X_3 = X_4 = Y_1 = Y_2 = +$. In particular, there are cocircuits $W$ and $Z$ such that $W^0 = \{1,3\}$, $Z^0 = \{2,4\}$, and $W_2 = W_4 = Z_1 = Z_3 = +$. Consider the region, $D = s_1^+ \cap s_2^+ \cap s_3^+ \cap s_4^+ \subseteq \mcP(\mcM)$.  This is the quadrilateral region bounded by cocircuits $X$, $Y$, $Z$, and $W$ in Figure \ref{fig:rank3-partial-drawing}.  

\begin{figure}[h]
\begin{center}
\includegraphics{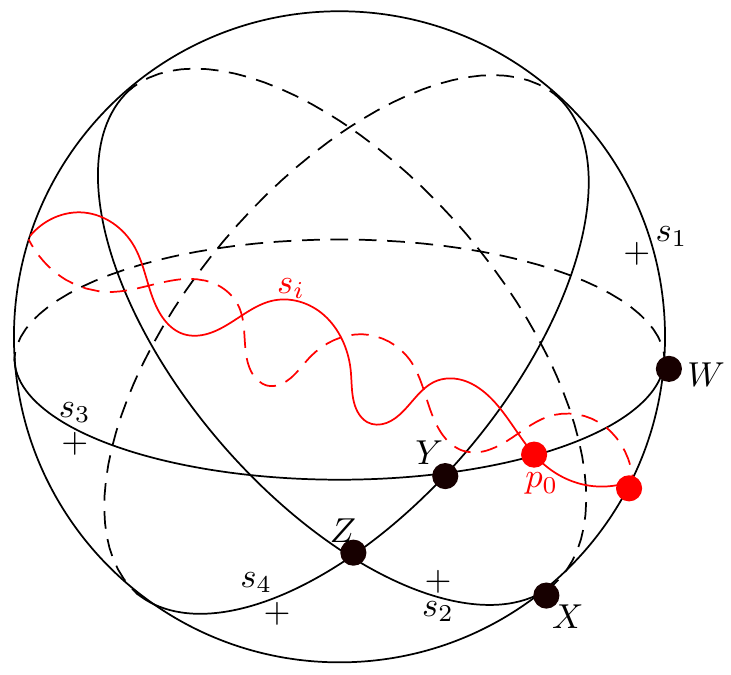}
\caption{The unique rank-$3$ pseudosphere arrangement with four pseudolines.}
\label{fig:rank3-partial-drawing}
\end{center}
\end{figure}

We claim that for each $i >4$, the pseudosphere $s_i$ can intersect the boundary of $D$ in at most two points.  Indeed, suppose $s_i$ intersects the boundary of $D$ at a point $p_0 \in s_j$ for some $j \in \{1,2,3,4\}$.  Because $\mcM$ is uniform, $p_0 \notin \{X,Y,Z,W\}$, so $s_j$ is unique. Let $\varphi_i: [0,1] \rightarrow \mcP(\mcM)$ be a parametrization of $s_i$.  We can assume $\varphi_i(0) = p_0$ and $\varphi_i(t)$ passes into the interior of $D$ for sufficiently small $t>0$.  Let $t_1$ be the next time when $\varphi_i(t_1)$ is on the boundary of $D$.  Assume $\varphi_i(t_1) \in s_k$.  Once again, $s_k$ is unique because $\mcM$ is uniform.  Further, $k \neq j$ because otherwise $s_j$ would intersect $s_i$ in at least four points: $\varphi_i(0)$, $\varphi_i(t_1)$, and their antipodes. 

When $t>0$ is sufficiently small, $\varphi_i(t) \in s_j^+ \cap s_k^+$.  When $t > t_1$ and $t-t_1$ is sufficiently small, $\varphi_i(t) \in s_j^+ \cap s_k^-$. By the definition of a pseudosphere arrangement, the image of $\varphi_i$ cannot cross back into $s_k^+$ before it crosses into $s_j^-$.  However, any other points where the image of $\varphi_i$ could intersect the boundary of $D$ lie in $s_j^+ \cap s_k^+$.  Thus $\varphi_i(0)$ and $\varphi_i(t_1)$ are the only points of intersection of $s_i$ with the boundary of $D$. 

Now we consider two paths from $X$ to $Y$ in $G^*(\mcM)$.  The first path $P_W$ travels from $X$ to $W$ along $s_1$, then from $W$ to $Y$ along $s_3$.  The second path $P_Z$ travels from $X$ to $Z$ along $s_2$, then from $Z$ to $Y$ along $s_4$.  Let $\ell(P_W)$ and $\ell(P_Z)$ denote the lengths of these paths. Initially, in $\mcM'$, $\ell(P_W) = \ell(P_Z) = 2$.

For each $i > 4$, the pseudosphere $s_i$ meets the boundary of $D$ in at most two points. This means $\ell(P_W) + \ell(P_Z)$ increases by at most two when we add $s_i$ back into $\mcP(\mcM)$.  Thus, in $\mcM$,
$$
\ell(P_W) + \ell(P_Z) \leq 4 + 2(n-4) = 2n-4.
$$
By the pigeonhole principle, either $\ell(P_W) \leq n-2$ or $\ell(P_Z) \leq n-2$,  so $d_{\mcM}(X,Y) \leq n-2$. 
\end{proof}

\begin{corollary} \label{cor:Hirsch-diam-contract-to-rank3}
Let $r \geq 3$ and $\mcM \in UOM(n,r)$. If $X,Y \in \mcC^*(\mcM)$ and $|X^0 \setminus Y^0| = 2$, then $d_{\mcM}(X,Y) \leq n-r+1$.
\end{corollary}

\begin{proof}
Let $A = X^0 \cap Y^0$.  Let $\{s_e \ : \ e \in E(\mcM)\}$ be the pseudospheres in the Folkman-Lawrence representation of $\mcM$ and let $S_A = \bigcap_{e \in A}s_e$.  Because $\mcM$ is uniform, $|A| = r-3$ and hence $S_A \approx \bbS^2$ is the Folkman-Lawrence representation of the uniform oriented matroid $\mcM/A \in UOM(n-r+3,3)$. 

Both $X$ and $Y$ are cocircuits on $S_A$ and clearly $X \neq -Y$, so by Theorem \ref{thm:Hirsch-UOM-rank3}, $$d_{\mcM}(X,Y) \leq d_{\mcM/A}(X,Y) \leq (n-r+3)-2 = n-r+1.$$ 

\end{proof}
Santos (personal communication) has pointed out that the proof of Theorem \ref{thm:Hirsch-UOM-rank3} cannot be directly extended to establish Conjecture \ref{conj:mainproblem} in rank four. For a realizable uniform oriented matroid of rank four, six hyperplanes will enclose a combinatorial cube.  For concreteness, we can consider the cube with $-1 \leq x_i \leq 1$ for all $i = 1,2,3$.

Figure \ref{paco-example} illustrates three edge-disjoint paths, colored red, green, and blue, from $(-1,-1,-1)$ to $(1,1,1)$.  Here, $(-1,-1,-1)$ is the vertex incident to the three dotted edges, and $(1,1,1)$ is its polar opposite.  The three images show slices of the cube by hyperplanes $x_i + x_j = (2-\varepsilon_k)x_k$ for all choices of $\{i,j,k\} = \{1,2,3\}$ and with $\varepsilon_1$, $\varepsilon_2$, and $\varepsilon_3$ all distinct. Each plane intersects two edges incident to $(-1,-1,-1)$ and two edges incident to $(1,1,1)$, and hence increases the total length of all three paths by at least four. If each of the remaining $n-6$ hyperplanes has one of the three illustrated types (with the $\varepsilon_k$ generic) then the total length of the red, blue, and green paths will be at least $4(n-6)+9$.  If there are approximately $\frac{n-6}{3}$ hyperplanes of each type, then each of the red, green, and blue paths will have length at least $\left\lfloor \frac{4}{3}n\right\rfloor -5$. 

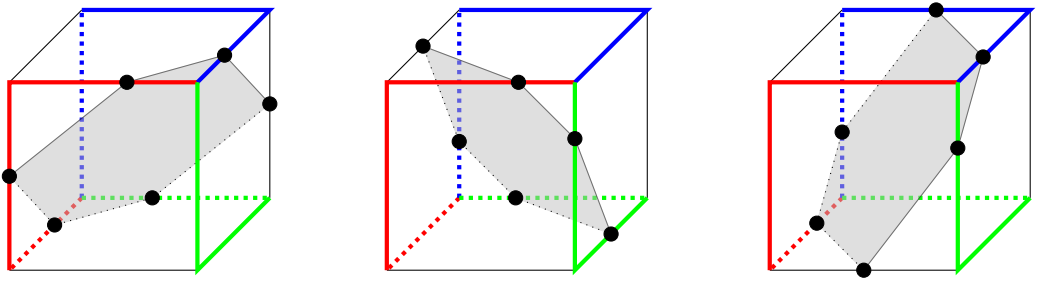
\begin{figure}[h]
\begin{center}
\begin{minipage}{.3\textwidth}
\begin{tikzpicture}[scale=1.25]
\draw[ultra thick, red, dotted] (-1,-1,-1) -- (-1,-1,1);
\draw[ultra thick, blue, dotted] (-1,-1,-1) -- (-1,1,-1);
\draw[ultra thick, green, dotted] (-1,-1,-1) -- (1,-1,-1);
\draw[fill=lightgray, lightgray, opacity=.5] (1,0,-1) --  (-.25,-1,-1) -- (-1,-1,-.25) -- (-1,0,1) -- (.25,1,1) -- (1,1,.25) -- (1,0,-1);
\draw[dotted] (1,0,-1) -- (-.25,-1,-1) -- (-1,-1,-.25) -- (-1,0,1);
\draw[opacity=.5]  (-1,0,1) -- (.25,1,1) -- (1,1,.25) -- (1,0,-1) ;
\draw (1,1,-1) -- (1,-1,-1);
\draw (1,-1,1) -- (-1,-1,1);
\draw (-1,1,1) -- (-1,1,-1);
\draw[ultra thick, red] (-1,-1,1) -- (-1,1,1) -- (1,1,1);
\draw[ultra thick, blue] (-1,1,-1) -- (1,1,-1) -- (1,1,1);
\draw[ultra thick, green] (1,-1,-1) -- (1,-1,1) -- (1,1,1);

\draw[fill=black] (-.25,-1,-1) circle (.075);
\draw[fill=black] (-1,-1,-.25) circle (.075);
\draw[fill=black] (.25,1,1) circle (.075);
\draw[fill=black]  (1,1,.25) circle (.075);
\draw[fill=black]  (1,0,-1) circle (.075);
\draw[fill=black]  (-1,0,1) circle (.075);
\end{tikzpicture}
\end{minipage}
\begin{minipage}{.3\textwidth}
\begin{tikzpicture}[scale=1.25]
\draw[ultra thick, red, dotted] (-1,-1,-1) -- (-1,-1,1);
\draw[ultra thick, blue, dotted] (-1,-1,-1) -- (-1,1,-1);
\draw[ultra thick, green, dotted] (-1,-1,-1) -- (1,-1,-1);
\draw[fill=lightgray, lightgray, opacity=.5] (1,-1,0) -- (-.4,-1,-1) -- (-1,-.4,-1) -- (-1,1,0) -- (.4,1,1) -- (1,.4,1) -- (1,-1,0);
\draw[dotted] (1,-1,0) -- (-.4,-1,-1) -- (-1,-.4,-1) -- (-1,1,0);
\draw[opacity=.5] (-1,1,0) -- (.4,1,1) -- (1,.4,1) -- (1,-1,0);
\draw (1,1,-1) -- (1,-1,-1);
\draw (1,-1,1) -- (-1,-1,1);
\draw (-1,1,1) -- (-1,1,-1);
\draw[ultra thick, red] (-1,-1,1) -- (-1,1,1) -- (1,1,1);
\draw[ultra thick, blue] (-1,1,-1) -- (1,1,-1) -- (1,1,1);
\draw[ultra thick, green] (1,-1,-1) -- (1,-1,1) -- (1,1,1);

\draw[fill=black] (1,-1,0) circle (.075);
\draw[fill=black]  (-1,1,0) circle (.075);
\draw[fill=black]  (-.4,-1,-1) circle (.075);
\draw[fill=black]  (.4,1,1) circle (.075);
\draw[fill=black]  (1,.4,1) circle (.075);
\draw[fill=black]  (-1,-.4,-1) circle (.075);
\end{tikzpicture}
\end{minipage}
\begin{minipage}{.3\textwidth}
\begin{tikzpicture}[scale=1.25]
\draw[ultra thick, red, dotted] (-1,-1,-1) -- (-1,-1,1);
\draw[ultra thick, blue, dotted] (-1,-1,-1) -- (-1,1,-1);
\draw[ultra thick, green, dotted] (-1,-1,-1) -- (1,-1,-1);
\draw[fill=lightgray, lightgray, opacity=.5] (0,1,-1) -- (-1,-.3,-1) -- (-1,-1,-.3) -- (0,-1,1) -- (1,.3,1) -- (1,1,.3) -- (0,1,-1);
\draw (1,1,-1) -- (1,-1,-1);
\draw (1,-1,1) -- (-1,-1,1);
\draw (-1,1,1) -- (-1,1,-1);

\draw[dotted] (0,1,-1) -- (-1,-.3,-1) -- (-1,-1,-.3) -- (0,-1,1);
\draw[opacity=.5] (0,-1,1) -- (1,.3,1) -- (1,1,.3) -- (0,1,-1);

\draw[ultra thick, red] (-1,-1,1) -- (-1,1,1) -- (1,1,1);
\draw[ultra thick, blue] (-1,1,-1) -- (1,1,-1) -- (1,1,1);
\draw[ultra thick, green] (1,-1,-1) -- (1,-1,1) -- (1,1,1);
\draw[fill=black] (0,1,-1) circle (.075);
\draw[fill=black] (0,-1,1) circle (.075);
\draw[fill=black] (-1,-.3,-1) circle (.075);
\draw[fill=black] (1,.3,1) circle (.075);
\draw[fill=black] (1,1,.3) circle (.075);
\draw[fill=black] (-1,-1,-.3) circle (.075);

\end{tikzpicture}
\end{minipage}
\end{center}

\caption{Hyperplanes $x_i+x_j = (2-\varepsilon_k)x_k$ slicing the $\pm1$ cube for $\{i,j,k\} = \{1,2,3\}$.}
\label{paco-example}
\end{figure}

\subsection{Results in low corank}\hfill

Recall that the corank of an oriented matroid of rank $r$ on $n$ elements is equal to $n-r$.

\begin{theorem}\label{corank-k-theorem}
Let $\mcM$ be a uniform oriented matroid of corank $k$.  Then $$\diam(G^*(\mcM)) \leq \max \{\diam(G^*(\mcM')): \mcM' \in UOM(r'+k, r'), \ 2\leq r'\leq k+2\}.$$  
\end{theorem}
\begin{proof}
Let $ \mcM$ be a uniform oriented matroid of corank $k$, and let $X, Y \in \mcC^*(\mcM)$ such that
$\diam(G^*(\mcM))=d_{\mcM}(X, Y)$. If  $Y =-X$, we are done,
since by, 
Lemma  \ref{lem:distance-lower-bound} the diameter of any uniform oriented matroid of corank $k$ is at least $k+2$, and 
 $d_{\mcM}(X, -X) = k+2$.
So we assume that $Y \neq -X$.

Consider the contraction  
$\mcM'=\mcM / (X^0 \cap Y^0)$, and let $X'$ and $Y'$ be the images of $X$ and $Y$ under this contraction. Let $r'=\rank(\mcM')$ and $n'=|E(\mcM')|$. We know that $\mcM'$ is uniform because $\mcM $ is.
Note that $(X')^0 \cap (Y')^0 = \emptyset$ by construction, so $\supp(X') \, \cup \, \supp(Y') = E(\mcM')$. In addition, since 
$\mcM'$ is uniform, $|\supp(X')|=|\supp(Y')|=k+1$. This shows $|E(\mcM')| \leq 2(k+1)$. Further, $\supp(X') \neq  \supp(Y')$ because 
$Y \neq -X$, so $|\supp(X') \, \cup \, \supp(Y') | \geq k+2$, which implies $2 \leq r' \leq k+1$, as $|E(\mcM')|=r'+k$.

Then, as
 $X', Y' \in \mcC^*(\mcM')$ and
 $G^*(\mcM')$ is a subgraph of  $G^*(\mcM)$, 
we have  $ \diam(G^*(\mcM))=d_{\mcM}(X, Y)\leq d_{\mcM'}(X', Y') \leq  \diam(G^*(\mcM'))$. Thus, we conclude that for every matroid $\mcM$ of corank $k$, there exists a matroid
$\mcM'\in UOM(r'+k, r')$,  where $2 \leq r'  \leq k + 2$,   such that  $\diam(G^*(\mcM)) \leq \diam(G^*(\mcM'))$. 
\end{proof}

\begin{theorem}\label{corank-4-theorem}
Let $ \mcM \in UOM(n,r)$ with  $n-r  \leq 4$.  Then  $\diam(G^*(\mathcal{M})) = n-r+2$.
\end{theorem}
\begin{proof} 
If  $n-r \leq 3$ the theorem  follows directly from 
Theorem \ref{corank-k-theorem} and Theorem \ref{thm:computational-result}.
                            
When $n-r =4$, by Theorem \ref{corank-k-theorem} we have $\Delta(r+4,r) \leq \max_{2 \leq r' \leq 6}\{\Delta(r'+4,r') \}$. 
However, by Theorem \ref{thm:computational-result}, for $ 2 \leq r' \leq 5$, $\max\{\Delta(r'+4,r')\} \leq r'+4-r'+2=6$. 
 So we only need to consider
$\mcM \in UOM(10,6)$.  Let $X,Y \in \mcC^*(\mcM)$ be such that $\diam(G^*(\mcM))=d_{\mcM}(X,Y)$. If $Y = -X$, the result holds by Lemma \ref{lem:distance-lower-bound}. If $X^0 \cap Y^0 \neq \emptyset$, then as in Theorem \ref{corank-k-theorem}, the contraction $\mcM' = \mcM/(X^0 \cap Y^0)$ satisfies $d_{\mcM}(X, Y)\leq \diam (\mcM')$. Since $|E(\mcM')| \leq 9$, the result holds by Theorem \ref{thm:computational-result}. So we may assume that $X^0 \cap Y^0 = \emptyset$.

Define $\mcT=X \circ Y$. Then, by Lemma \ref{lem:tope_AP_equiv}
the graph $ G(\mcT)$
of $\mcT$ is isomorphic to the graph $G_A(\mcA)$ of $\mcA$, where $\mcA$ is the abstract polytope on the covector of $\mcT$
with dimension $5$ on 10 elements. However, 
by \cite[Theorem  7.1]{ad}   the diameter   of $G_A(\mcA)$ is $5$, implying  that
$d_{\mcM}(X,Y) = 5$.
Noting that  $d_{\mcM}(X,-X) =6$, we conclude that $\diam(G^*(\mathcal{M})) = 6$ which completes the proof. 
\end{proof}

Note that while the theorems about coranks in this subsection are for uniform oriented matroids, they are valid for general oriented matroids due to Lemma \ref{perturbation-lemma}. Now we are ready to combine all the results in this section to prove Theorem \ref{thm:lowrankcorank}.

\begin{proof} (of Theorem \ref{thm:lowrankcorank})

The proof of part (a) for small oriented matroids is in Theorem \ref{thm:computational-result}. The proof of part (b) for rank three oriented matroids is in Theorem \ref{thm:Hirsch-UOM-rank3}. The proof of part (c) for oriented matroids of corank no more than four is in Theorem \ref{corank-4-theorem}.
    
\end{proof}


\section{An Improved Quadratic Diameter Bound and Two Conjectures}
\label{section:quadraticbound}

Let $\mcM$ be an oriented matroid.  Recall that a \emph{coline} in $\mcM$ is a one-dimensional sphere in the Folkman-Lawrence representation of $\mcM$. Now we present an improved quadratic upper bound on $\Delta(n,r)$ for uniform oriented matroids, improving Theorem \ref{thm:Finschi-bound}. In particular, our Eq.~\ref{quadbound-part1} is tight for rank three.  Our proof relies on a modification of Finschi's proof \cite{Finschi-thesis}.


\begin{proof} (of Theorem \ref{thm:Finschi-improved})

By Lemma \ref{perturbation-lemma}, it suffices to consider the case that $\mcM$ is uniform. We prove the claim by induction on $|X^0 \setminus Y^0|$.  If $|X^0 \setminus Y^0| = 1$, then $d_{\mcM}(X,Y) \leq n-r+1$ by Lemma \ref{lem:distance-lower-bound}. If $|X^0 \setminus Y^0| = 2$, then $d_{\mcM}(X,Y) \leq n-r+1$ by Corollary \ref{cor:Hirsch-diam-contract-to-rank3}.

Now we move on to the inductive step.  Suppose $ |X^0 \setminus Y^0| = \ell \geq 3$. Pick any element $e \in Y^0 \setminus X^0$, and consider the coline $U$, with $U^0 = Y^0 \setminus \{e\}$. Note that $|U^0 \setminus X^0| = \ell-1$.

Now we look more carefully at the coline $U$, which is a cycle on $2(n-r+2)$ cocircuits. We distinguish $\ell$ pairs of these cocircuits.  For each element $f \in X^0 \setminus U^0$, there is a cocircuit $Z^f$ with $(Z^f)^0 = U^0 \cup \{f\}$.  Because $|X^0 \setminus U^0| = \ell$, there are $\ell$ such pairs of antipodal cocircuits, which we denote as $\pm Z^1, \ldots, \pm Z^\ell$ for simplicity.    

The cocircuits $Y$ and $-Y$ are antipodal on $U$, and hence partition $U$ into two halves, each of which contains $n-r+1$ cocircuits. Assume without loss of generality that  $Z^1,\ldots,Z^\ell$ all lie on one half of the coline (as it is partitioned by $Y$ and $-Y$), and further that $Z^1,\ldots,Z^\ell$ are ordered by their distance from $Y$, with $Z^1$ closest to $Y$ and $Z^\ell$  farthest. 

Because there are $(n-r+2)-(\ell+1) = n-r-\ell+1$ remaining pairs of antipodal circuits on $U$, and at most one element from each pair can lie on the arc from $Z^1$ to $-Z^\ell$ that contains $Y$, it follows that there exists a path of length at most $\left \lfloor \frac{n-r-\ell+1}{2} \right\rfloor + 1$ from $Y$ to one of $Z^1$ or $-Z^\ell$ along $U$. For simplicity, let $Z$ denote whichever of $Z^1$ and $-Z^\ell$ is closer to $Y$ along $U$. 

In summary, we have shown that there exists a cocircuit $Z$ whose distance to $Y$ is at most $\left\lfloor \frac{n-r-\ell+1}{2}\right\rfloor + 1$ with $|X^0 \setminus Z^0| = \ell-1$.  Because $\ell-1 \neq 0$, we know $Z \neq -X$ as well. The result now follows by induction, and after reindexing with $k=\ell-1$ we have
\begin{equation*} 
d_{\mcM}(X,Y) \leq n-r+1 + \sum_{k=2}^{|X^0 \setminus Y^0|-1}\left(\left \lfloor \frac{n-r-k}{2} \right\rfloor + 1\right).
\end{equation*}

 To get Eq.~\eqref{quadbound-part2}, note that $|X^0 \setminus Y^0| \leq \min (r-1,n-r+1)$, because $|X^0 \setminus Y^0| \leq |X^0| = r-1$ and $|X^0 \setminus Y^0| \leq |E\setminus Y^0| = n-r+1$. So, when $r \geq 4$ and $n-r \geq 2$, 
\begin{equation*} \label{quadbound-part2}
\diam(G^*(\mcM)) \leq n-r+1 + \sum_{k=2}^{\min(r-2,n-r)} \left(\left\lfloor \frac{n-r-k}{2}\right\rfloor +1 \right).
\end{equation*}
%
\end{proof}


We end with the fascinating Conjecture \ref{London-Paris-Conj}. Let $\mcM$ be a uniform oriented matroid. We say a path $X^1, X^2, \ldots, X^k$ in the cocircuit graph $G^*(\mcM)$ \textit{stays on a tope} $\mcT$ if each cocircuit $X^i$ is a vertex of $\mcT$. If Conjecture \ref{London-Paris-Conj} were true, then Theorem \ref{thm:Finschi-improved} would imply a quadratic upper bound on the diameter of any polytope, proving the polynomial Hirsch Conjecture. The famous polynomial Hirsch conjecture states that the diameter of all convex polytopes is bounded by a polynomial in terms of the number of facets and the dimension (see \cite{Santos-recent}).  A modification of the computer code used in Section \ref{section:smallmatroids} shows that Conjecture \ref{London-Paris-Conj} holds for all uniform oriented matroids with at most nine elements. The code is included on Github as well.

%
A. Adler has pointed out that the analogous conjecture to Conjecture \ref{London-Paris-Conj} for polytopes is false: it is possible to have two vertices on a common facet while the shortest path between them leaves the facet. It is interesting to note that it has been shown  in \cite{Eisenbrandetal2010} some abstractions of convex polytopes provide an almost quadratic lower bound, but we do not know of any possible connections to oriented matroids. The validity of Conjecture \ref{London-Paris-Conj} implies a quadratic bound on the diameter of all polyhedra, while the validity of Conjecture \ref{conj:mainproblem} gives a linear diameter to all oriented matroids. However, Conjecture \ref{conj:mainproblem} and Conjecture \ref{London-Paris-Conj} cannot both be true. 
 
 \begin{proposition}
 Conjecture \ref{conj:mainproblem} and Conjecture \ref{London-Paris-Conj} cannot both be true for all $n$ and $r$.
 \end{proposition}
 
 \begin{proof}
 Santos \cite[Theorem 1.8]{Santos} gave examples of simple polytopes of diameter at least $\frac{21}{20}(n-d)$ when $n$ and $d$ are sufficiently large.  In particular, there is eventually a simple polytope $P$ with diameter at least $n-d+4$.  Let $\mcM$ be its corresponding uniform oriented matroid. If Conjecture \ref{London-Paris-Conj} were true, then $\diam(G^*(\mcM)) \geq n-r+3$.  Similarly, if Conjecture \ref{conj:mainproblem} were true, then $G^*(\mcM)$ would necessarily have a shorter path than the one on tope $P$ between the vertices at distance $n-d+4 = n-r+3$. Thus, both conjectures cannot hold simultaneously or we reach a contradiction.
 \end{proof}
 
 We conclude with a strengthening of Conjecture 1.9 using crabbed paths.

 \begin{conjecture}\label{Ilan-stronger-Conj}
Let $\mcM$ be a uniform oriented matroid, and let $X,Y \in \mcC^*(\mcM)$ be cocircuits such that $S(X,Y)=\emptyset$.
Then, there exists a crabbed path from $X$ to $Y$ whose length is no bigger than the length of any path from $X$ to $Y$ in $\mcM$.
\end{conjecture}

Conjecture \ref{Ilan-stronger-Conj} implies Conjecture \ref{London-Paris-Conj}: if there is a crabbed path that is no longer than the shortest path between the two cocircuits, then the diameter computed over the topes that contain $X,Y$ is always no larger than the diameter of the entire cocircuit graph. We do not have as much evidence to support Conjecture \ref{Ilan-stronger-Conj}, but due to its remarkable implications if true it is certainly worthy of additional study.
  
\medskip
\noindent {\bf Acknowledgements:} The first author was supported by The Tsinghua-Berkeley Shenzhen Institute. The second and the fourth author were supported by NSF grant DMS-1818969. The third author was supported by NSF grant DMS-1600048. And the fourth author was also supported by NSF HDR TRIPODS grant CCF-1934568. We are grateful for the comments and suggestions we received from Aviv Adler, Louis Billera, Lukas Finschi,
Komei Fukuda, Kolja Knauer, Nati Linial, Francisco Santos, and Tam\'as Terlaky.

\bibliographystyle{plainurl}
\bibliography{UOMbib3}
\end{document}